\documentclass[11pt]{article}
\usepackage[numbers,sort&compress]{natbib}
\usepackage{color,colordvi}
\usepackage{fullpage}

\usepackage{booktabs}
\usepackage{caption}
\usepackage{amsmath}
\usepackage{amsfonts,amsthm,amssymb,mathrsfs,bbding}
\usepackage{float}
\usepackage{graphicx}
\usepackage{enumerate}
\usepackage{tikz}
\usepackage{caption}
\usepackage{rotating}
\allowdisplaybreaks[4]
\usepackage{tabularx}

\newtheorem{thm}{Theorem}

\newtheorem{lem}{Lemma}[section]
\newtheorem{cor}{Corollary}
\newtheorem{pro}{Proposition}[section]

\newtheorem{defi}{Definition}[section]
\renewcommand\proofname{\bf Proof}

\begin{document}

\title{\bf  Minimally $(k,k)$-edge-connected graphs via spectral radius}
\author{\bf Yu Wang$^a$,\  Dan Li$^a$, \ Huiqiu Lin$^{a,b}$\thanks{Corresponding author. E-mail address: huiqiulin@126.com (H. Lin)}\\[2mm]
\small\it $^a$College of Mathematics and System Science, Xinjiang
University, \\
\small\it   Urumqi 830017, PR China\\[1mm]
\small\it  $^b$Department of Mathematics, East China University of Science and Technology,  \\
\small\it   Shanghai 200237, PR China}
\date{}
\small{}
\maketitle
{\flushleft\large\bf Abstract}:
For $l > 1$, the $l$-edge-connectivity $\kappa'_l(G)$ of a connected graph $G$ is defined as the minimum number of edges whose removal leaves a graph with at least $l$ components. A graph is minimally $(k,l)$-edge-connected if $\kappa'_l(G)\geq k$ but for any edge $e\in E(G)$ satisfies that $\kappa'_l(G-e)< k$. Motivated by two foundational extremal problems: Brualdi and Solheid's problem [SIAM J. Algebra Discrete Methods (1986)] for graphs of fixed order: determine sharp upper bounds for the spectral radius over graph families and characterize extremal graphs; and its fixed size analogue proposed by Brualdi and Hoffman [Linear Algebra Appl. (1985)], we resolve both problems for minimally $(k,k)$-edge-connected graphs. Building on the structural framework of Hennayake, Lai, Li, and Mao [J. Graph Theory (2003)], we combine edge-switching method and double eigenvectors skill to characterize the graphs maximizing the spectral radius among all minimally $(k,k)$-edge-connected graphs of prescribed order or size. Our results generalize the $k=2$ cases established by Lou, Min, and Huang [Electron. J. Comb. (2023)] and Chen and Guo [Discrete Math. (2019)].

\maketitle {\flushleft\textit{\bf Keywords}:
Spectral radius; minimally $(k,l)$-edge-connected; edge-connected.}

\section{Introduction}
\vspace{1ex}

Throughout, we consider simple graphs and use $\kappa'(G)$ to denote the \emph{edge-connectivity} of a graph $G$. 
For an integer $l\geq2$, Boesch and Chen \cite{F.T} defined the \emph{$l$-edge-connectivity} $\kappa'_l(G)$ of a connected graph $G$ to be the smallest number of edges whose removal leaves a graph with at least $l$ components, if $|V(G)|\geq l$; otherwise, $\kappa'_l(G)=|V(G)|$. 
As a generalization of $l$-edge-connectivity, Oellermann \cite{O.R} introduced the concept of \emph{$(k,l)$-edge-connected}, defined by $\kappa'_l(G)\geq k$.  A graph $G$ is \emph{minimally $(k,l)$-edge-connected} if $\kappa'_l(G)\geq k$ but for any edge $e\in E(G)$ satisfies that $\kappa'_l(G-e)< k$. Thus, a minimally $(2,2)$-edge-connected graph is  equivalent to a minimally $2$-edge-connected graph. 

For a graph $G$, let $A(G)$ be its adjacency matrix, and let $\lambda_i(G)$ denote the $i$th largest eigenvalue of $A(G)$.
Particularly, the largest eigenvalue of $A(G)$, denoted by $\rho(G)$, is called the \textit{spectral radius} of $G$. Based on the above facts, extensive research has explored the connections between edge-connectivity and graph eigenvalues. In 2004, Chandran \cite{S.L.} proved that 
$\kappa'(G)=d$ if $G$ is an $n$-vertex
$d$-regular graph with $\lambda_2(G)<d-1-\frac{d}{n-d}$. This result was extended by Cioab\u{a} \cite{S.M.}, who  showed that for $d$-regular graph, $\kappa'(G)\geq d$ proveded that $\lambda_2(G)<d-\frac{2(\kappa'(G)-1)}{d+1}$. Gu, Lai, Li and Yao\cite{Gu} subsequently generalized these findings to general graphs,
proving that $\kappa'(G)\geq k$ when $\lambda_2(G)<\delta-\frac{2(\kappa'(G)-1)}{\delta+1}$, where the minimum degree satisfies
$\delta\geq 2\kappa'(G)\geq 4$.

A fundamental problem concerning spectral radius was posed by Brualdi and Solheid \cite{R.A.B}: Given a set of graphs $\mathscr{G}$, find
an upper bound for the spectral radius of graphs in $\mathscr{G}$ and characterize the graphs in which the maximal spectral
radius is attained. 
Recent investigations have increasingly focused on the interplay between graph structure and edge-connectivity. Notably, Ning, Lu and Wang \cite{W.N} began exploring the edge-connectivity of a graph in relations to its spectral radius, conjecturing that the graph $F^{\kappa'}_{n,\delta+1}$ maximizes the spectral radius among all $n$-vertex graphs with minimum degree $\delta$ and edge-connectivity $\kappa'$. Here, $F^{\kappa'}_{n,\delta+1}$ denotes the graph obtained from $K_{\delta+1}\cup K_{n-\delta-1}$
by adding $\kappa'$ edges joining a vertex in $K_{\delta+1}$ and $\kappa'$ vertices in $K_{n-\delta-1}$. Fan, Gu and Lin \cite{D.D.} later verified this conjecture for $n\geq 2\delta+4$. Furthermore, Chen and Guo \cite{X.D.} determined the maximum spectral radius for the minimally $(2,2)$-edge-connected graphs with
order $n$ and characterized the corresponding extremal graphs. Recently, Lou and He \cite{Z.ZL} proved that for any fixed integer $k\geq 3$, setting $\alpha=\frac{1}{24k(k+1)}$, the complete bipartite graph $K_{k,n-k}$ uniquely achieves the maximal spectral radius among all $n$-vertex minimally $(k,2)$-edge-connected graphs when  $n\geq\frac{18k}{\alpha^2}$. The relation between the edge-connectivity $\kappa'(G)$ and eigenvalues was also well investigated by many researchers\cite{D.F.,Y.W}.
Motivated by these advances, in this paper, we 
obtained the corresponding extremal graphs in 
minimally $(k,k)$-edge-connected graphs of order $n$ with the maximum spectral radius.
Let $K^{k-2}_{2, n-k}$ be a graph obtained by identifying the $k\!-\!2$ degree vertex and the $n-k$ degree vertex from the disjoint union of $K_{1,k-2}$ and $K_{2,n-k}$.

\begin{thm}\label{thm::1}
Let $G$ be a minimally $(k,k)$-edge-connected graph of order $n\geq 4k$. Then
$$\rho(G) \leq \frac{\sqrt{4n-2k-4+2\sqrt{5k^2-8kn+4n^2-4k+4}}}{2}$$ with equality if and only if $G \cong K^{k-2}_{2, n-k}$.
\end{thm}
Analogous to upper bound on spectral radius in terms of order $n$. In the 1985s,  Brualdi and Hoffman \cite{R.A} established an upper bound on spectral radius in terms of size $m$: if $m \leq\binom{k}{2}$ for some integer $k \geq 1$, then $\rho(G) \leq k-1$, with equality if and only if $G$ consists of a $k$-clique and isolated vertices. Building upon this foundation, Stanley \cite{R.P} showed that $\rho(G) \leq \frac{\sqrt{1+8 m}-1}{2}$. Subsequent investigations have focused on graphs with specific combinatorial structures. For example,  Nosal \cite{E.N} proved that if $G$ is a triangle-free graph with $m$ edges, then $\rho(G) \leq \sqrt{m}$. More recently, Lin, Ning and Wu \cite{H.L} proved that $\rho(G) \leq \sqrt{m-1}$ when $G$ is a non-bipartite and triangle-free graph of size $m$. Zhai, Lin and Shu \cite{M.Q.} obtained that if $G$ contains neither pentagon nor hexagon with a given size $m$, then $\rho(G) \leq \frac{1}{2}+\sqrt{m-\frac{3}{4}}$, with equality holding if and only if $G$ is a book graph. Notably, Lou, Min and Huang \cite{Z.Z} determined the maximum spectral radius for the minimally $(2,2)$-edge-connected graphs with size $m$ and characterized the corresponding extremal graphs. In 2025, Zhai, Lin and Shu \cite{M.Q} determined the maximum spectral radius among all the minimally $(k,2)$-edge-connected graphs with size $m\geq \Omega(k^6)$ and characterized the corresponding extremal graphs. 

Motivated by those researches, we aim  to study the spectral extremal problem of minimally $(k,k)$-edge-connected graph with size $m$. 
Let $K^{k-2}_{2, \frac{m-k+2}{2}}$ be a graph obtained by identifying the $k-2$ degree vertex of $K_{1,k-2}$ and the $\frac{m-k+2}{2}$ degree vertex of $K_{2,\frac{m-k+2}{2}}$.
Let $K^{k-2}_{2,\frac{m-k-1}{2}}*K_3$ be a graph obtained from by identifying the $k-2$ degree vertex of $K_{1,k-2}$, the vertex of a triangle and the $\frac{m-k-1}{2}$ degree vertex of $K_{2,\frac{m-k-1}{2}}$. 

\begin{thm}\label{thm::2}
Let $G$ be a minimally $(k,k)$-edge-connected graph of size $m\geq (432 k)^2$.

\noindent $(i)$ If $m-k\equiv 0\ (mod\ 2)$, then $$\rho(G)\leq \frac{\sqrt{2m+2\sqrt{2k^2-2km+m^2-8k+4m+8}}}{2}$$ with equality if and only if $G\cong K^{k-2}_{2,\frac{m-k+2}{2}}$. 

\noindent $(ii)$ If $m-k\equiv 1\ (mod\ 2)$, then $$\rho(G)\leq \rho\big(K^{k-2}_{2,\frac{m-k-1}{2}}*K_3\big)$$ with equality if and only if $G\cong K^{k-2}_{2,\frac{m-k-1}{2}}*K_3$. 
\end{thm}
\section{ Preliminaries}
For any $v\in V(G)$, let $N_{G}(v)=\{u\mid  uv\in E(G)~\text{and}~u\in V(G)\}$ and $N_{G}[v]=N_{G}(v)\cup\{v\}$. Given a vertex subset $F \subset V(G)$, we denote by $G-F$  the graph obtained from $G$ by deleting the vertices of $F$ together with the edges incident with them.

\begin{defi}[Cvetovi\'{c}, Rowlinson, Simi\'{c} \cite{D.C}]
Given a graph $G$, the vertex partition $\Pi: V(G)=V_1\cup \cdots\cup V_k$ is said to be an \emph{equitable partition} if, for each $u\in V_i$, $|V_j\cap N_{G}(u)|=b_{ij}$ is a constant depending only on $i$, $j$ $(1\leq i, j\leq k)$. The matrix $B_\Pi=(b_{ij})$ is called the \emph{quotient matrix} of $G$ with respect to $\Pi$.
\end{defi}

\begin{lem}[Cvetovi\'{c}, Rowlinson, Simi\'{c} \cite{D.C}]\label{lem::2.1}
Let $\Pi: V(G)=V_1\cup \cdots\cup V_k$ be an equitable partition of $G$ with quotient matrix $B_\Pi$. Then $\det(xI-B_\Pi)\mid\det(xI-A(G))$. Furthermore, the largest eigenvalue of $B_\Pi$ is just the spectral radius of $G$.
\end{lem}

\begin{lem}[See \cite{R.A.}]\label{lem::2.2}
Let $G$ be a connected graph and $G'$ be a proper subgraph of $G$. Then $\rho(G')<\rho(G)$.
\end{lem}

\begin{lem}[See \cite{Z.W}]\label{lem::2.3}
Let $(H,v)$ and $(K,w)$ be two connected rooted graphs. Then 
$$\rho((H,v)*(K,w))\leq \sqrt{\rho(H)^2+\rho(K)^2},$$
the equality holds if and only if both $H$ and $K$ are stars, where $(H,v)*(K,w)$ is obtained by identifying $v$ and $w$ from disjoint union of $H$ and $K$.
\end{lem}

Edge-switching method is one of the most important tools in spectral graph theory. The following lemma states a widely used edge-switching operation.

\begin{lem}[See \cite{B.W}] \label{lem::2.4}
Let $G$ be a connected graph, and let $u$, $v$ be two vertices of $G$. Suppose $\{v_1, v_2, \ldots, v_s\} \subseteq N_{G}(v)
\setminus N_{G}(u)(1 \leq s \leq d_G(v))$ and  $G^*$ is the graph obtained from $G$ by adding the edges $uv_i$ and deleting
the edges $vv_i$, where $1 \leq i \leq s$. If $x_u \geq x_v$, then $\rho(G)<\rho(G^*)$.
\end{lem}

Our one primary tool is sharp upper bound on the spectral radius, which was obtained by Hong, Shu, and Fang \cite{Y.H} and Nikiforov \cite{V.N}, independently.

\begin{lem}[See \cite{Y.H} and \cite{V.N}]\label{lem::2.5}
Let $G$ be a graph on $n$ vertices and $m$ edges with minimum degree 
$\delta\geq 1$. Then $$\rho(G)\leq\frac{\delta-1}{2}+\sqrt{2m-n\delta+\frac{(\delta+1)^2}{4}},$$ 
with equality if and only if $G$ is a $\delta$-regular graph or a bidegreed graph in which each vertex is of degree either $\delta$ or $n-1$.
\end{lem}

\begin{lem}[See \cite{K.H}] \label{lem::2.6}
Let $k\geq 2$ be an integer, and let $G$ be a connected graph. The following are equivalent.

\noindent $(i)$ For any edge $e \in E(G)$, $\kappa'_k (G-e)=k-1$.

\noindent $(ii)$ $G$ is minimally $(k,k)$-edge-connected.
\end{lem}

\begin{lem}[See \cite{K.H}] \label{lem::2.7}
Let $G$ be a connected graph with a bridge set $B(G)$, and let $k\geq 2$ be an integer.  Each of the following holds.

\noindent $(i)$ If $G$ is $(k, k)$-edge-connected, then $|B(G)|\leq k-2$.

\noindent $(ii)$ If $G$ is minimally $(k, k)$-edge-connected, then for any $e \in E(G)$, $G$ has an edge subset $T \subseteq E(G)$ with $|T|=k$ and with $e \in T$, such that $B(G) \subset T$ and such that $G-T$ has $k$ components.

\noindent $(iii)$ If $G$ is $(k, k)$-edge-connected, then for any $T \subset E(G)$ with $B(G) \subset T$ and with $|T|=k$ such that $G-T$ has $k$ components, there exists exactly one $i$ with $1 \leq i \leq s$ such that $T-B(G) \subseteq E(H_i)$.

\noindent $(iv)$ Suppose that $v \in V(G)$ is a vertex of degree 1 in $G$. Then $G$ is minimally $(k, k)$-edge-connected if and only if $G-v$ is minimally $(k\!-\!1, k\!-\!1)$ edge-connected.
\end{lem}

A path $P$ in a graph $G$ is called an \emph{internal path} if all vertices of $P$ have degree two except for the two end vertices. Xue, Liu and Shu \cite{J.X} presented the unimodality of principal eigenvector on internal paths.

\begin{lem}  [See \cite{J.X}]\label{lem::2.8} 
Let $P:=v_0 v_1 \cdots v_k$ be an internal path in a connected graph $G$ and $x$ be a perron eigenvector of $G$. If $\rho(G)>2$, then the following statements hold:

\noindent (i) If $x_0=x_k$, then $x_0>x_1>\cdots>x_{\lfloor k/2\rfloor}=x_{\lceil k / 2\rceil}<\cdots<x_{k-1}<x_k$ and $x_i=x_{k-i}$ for $0 \leq i \leq k$.

\noindent (ii)  Suppose $x_k>x_0$. If $x_0 \leq x_1$, then $x_0 \leq x_1<x_2<\cdots<x_k$. If $x_0>x_1$, then $x_{k-i}>x_i$ for $1 \leq i \leq\lceil k / 2\rceil-1$, and there exists an integer $1 \leq j \leq\lceil k / 2\rceil-2$ such that $x_0>x_1>\cdots>x_j \geq x_{j+1}<\cdots<x_k$.
\end{lem}

As we know, a block of a graph is a maximal 2-connected subgraph with respect to vertices. A block of a graph is called \emph{leaf block} if it contains exactly one cut vertex.

\begin{lem}[See \cite{Z.Z}] \label{lem::2.9}
If $G$ is a minimally $(2, 2)$-edge-connected graph, then no cycle of $G$ has a chord.
\end{lem}

\begin{cor}\label{cor::1}
Let $G$ be a minimally $(k, k)$-edge-connected graph and $B\subset V(G)$. For $l\leq k$, if $G[B]$ is a block and it is a minimally $(l, l)$-edge-connected, then no cycle of $G$ has an internal path with length less than $l$.
\end{cor}
\begin{proof}
If not, then there exists a cycle in $G$ that contains an internal path $P$ with a length less than $l$. Let $P=v_0v_1\cdots v_p$ where $p< l$, and define $B'=G[B]-v_0v_1$. We have $\kappa'_{l}(B')=l-1$ by Lemma \ref{lem::2.6} and $E_1\subseteq B'$ according to $(iii)$ of Lemma \ref{lem::2.7}. Now suppose $E_1\cup v_0v_1$ is an $l$-edge cut of $G[B]$ with $|E_1|=l-1$. Every edge in $|E_1|$ is a cut edge of $B'$ since $B'$ is a connected graph. However, because $G[B]$ is a block and all vertices not in the internal path $P-v_0v_1$ have degree at least $2$, $B'$ cannot have $l-1$ cut-edges, a contradiction.
\end{proof} 

\begin{lem} [See \cite{W.M}]\label{lem::2.10}
Every minimally $(2,2)$-edge-connected graph of order $n\geq 3k-2$ has at
most $2(n-2)$ edges.
\end{lem}


\section{ Proof of Theorem \ref{thm::1}}
In this section, we present the proof of Theorem \ref{thm::1}. While Chen and Guo \cite{X.D.} established the case for $k=2$,  we focus here on the general case where $k\geq 3$.
Denote by $u_1$ a vertex of maximum degree of the friend graph with $t_1$ triangles, $u_2$ a vertex of maximum degree of $K_{2, t_2}$, $u_3$ a vertex of $K_{2, t_3+1}$ with degree 2, $u_4$ a vertex of maximum degree of $K_{1,k-2}$, and $u_5$ a vertex of maximum degree of $K_{1,k-3}$. Let $K_1(t_1, t_2, t_3)$ be the graph obtained from the above four graphs by identifying $u_1, u_2$, $u_3$, and $u_4$, and $K_2(t_1, t_2, t_3)$ be the graph obtained from the above four graphs by identifying $u_1, u_2$, $u_3$, and $u_5$, and  
$K(t_1, t_2)$ be the graph obtained from the above three graphs by identifying $u_1, u_2$ and $u_4$.

\begin{lem} \label{lem::3.1}
If $n\geq 4k$, then
$\rho(K^{k-2}_{2, n-k})=\frac{\sqrt{4n-2k-4+2\sqrt{5k^2-8kn+4n^2-4k+4}}}{2}$. 
\end{lem}
\begin{proof}
The vertex set of $K^{k-2}_{2, n-k}$ has an equitable partition $\pi$ and its corresponding quotient matrix is
$$B_\pi=\left(
           \begin{array}{cccc}
             0 & 0 & k-2 & n-k\\
             0 & 0 & 0  & n-k\\
             1 & 0 & 0  & 0\\
             1 & 1 & 0  & 0\\
           \end{array}
         \right).$$
Then $f(x)=\det(xI_4-B_\pi)=x^4+(k-2n+2)x^2-k^2+kn+2k-2n.$ By Lemma \ref{lem::2.1}, $\rho(K^{k-2}_{2, n-k})$ is the largest root of $f(x)=0$. After a simple calculation, we can get $\rho(K^{k-2}_{2, n-k})=\frac{\sqrt{4n-2k-4+2\sqrt{5k^2-8kn+4n^2-4k+4}}}{2}$.
\end{proof}

\begin{lem} \label{lem::3.2}
If $n\geq 4k$, then $\rho(K_i(\frac{n-k+i}{2},0,0))$ is the largest root of $x^3-x^2+(1-n)x+k-(i+1)=0$ and $\rho(K_i(\frac{n-k+i}{2},0,0))<\frac{\sqrt{4n-2k-4+2\sqrt{5k^2-8kn+4n^2-4k+4}}}{2}$ where $i=1$ or $2$. 
\end{lem}
\begin{proof}
The vertex set of $K_i(\frac{n-k+i}{2},0,0)$ has equitable partition $\pi$ and its corresponding quotient matrix is
$$B_\pi=\left(
           \begin{array}{ccc}
             0 & k-(i+1) & n-k+i \\
             1 & 0  & 0\\
             1 & 0 & 1 \\
           \end{array}
         \right).$$
Then $f(x)=\det(xI_3-B_\pi)=x^3-x^2+(1-n)x+k-(i+1).$ By Lemma \ref{lem::2.1}, $\rho(K_i(\frac{n-k+i}{2},0,0))$ is the largest root of $f(x)=0$. One can verify that 
$f(\frac{\sqrt{4n-2k-4+2\sqrt{5k^2-8kn+4n^2-4k+4}}}{2})>0$, and so $\rho(K(\frac{n-k+i}{2},0,0))<\frac{\sqrt{4n-2k-4+2\sqrt{5k^2-8kn+4n^2-4k+4}}}{2}$. 
\end{proof}

In order to provide the proof of Theorem \ref{thm::1}, we begin by proving the following two valuable propositions.

\begin{figure}[http]
\centering
\begin{tikzpicture}[x=1.00mm, y=1.00mm, inner xsep=0pt, inner ysep=0pt, outer xsep=0pt, outer ysep=0pt]
\path[line width=0mm] (59.21,36.55) rectangle +(108.79,40.65);
\definecolor{L}{rgb}{0,0,0}
\definecolor{F}{rgb}{0,0,0}
\path[line width=0.30mm, draw=L, fill=F] (78.34,59.05) circle (0.50mm);
\path[line width=0.30mm, draw=L, fill=F] (64.70,72.70) circle (0.50mm);
\path[line width=0.30mm, draw=L, fill=F] (71.12,72.70) circle (0.50mm);
\path[line width=0.30mm, draw=L, fill=F] (91.99,72.70) circle (0.50mm);
\path[line width=0.30mm, draw=L, fill=F] (85.57,72.70) circle (0.50mm);
\path[line width=0.21mm, draw=L] (65.10,72.78) -- (70.64,72.78);
\path[line width=0.21mm, draw=L] (85.97,72.78) -- (91.51,72.78);
\path[line width=0.21mm, draw=L] (64.70,72.78) -- (78.35,58.82);
\path[line width=0.21mm, draw=L] (71.36,72.78) -- (78.11,58.98);
\path[line width=0.21mm, draw=L] (85.33,72.86) -- (78.59,59.14);
\path[line width=0.21mm, draw=L] (91.67,72.30) -- (78.51,59.14);
\path[line width=0.30mm, draw=L, fill=F] (78.43,72.78) circle (0.15mm);
\path[line width=0.30mm, draw=L, fill=F] (92.07,67.73) circle (0.50mm);
\path[line width=0.30mm, draw=L, fill=F] (92.07,63.47) circle (0.50mm);
\path[line width=0.30mm, draw=L, fill=F] (92.07,56.57) circle (0.50mm);
\path[line width=0.30mm, draw=L, fill=F] (102.02,58.58) circle (0.50mm);
\path[line width=0.21mm, draw=L] (78.83,59.22) -- (91.83,67.40);
\path[line width=0.21mm, draw=L] (78.51,59.06) -- (91.99,63.15);
\path[line width=0.21mm, draw=L] (92.39,67.57) -- (101.86,58.74);
\path[line width=0.21mm, draw=L] (92.47,63.31) -- (101.86,58.66);
\path[line width=0.21mm, draw=L] (78.59,58.98) -- (91.67,56.57);
\path[line width=0.21mm, draw=L] (92.31,56.57) -- (101.86,58.58);
\path[line width=0.30mm, draw=L, fill=F] (68.07,52.15) circle (0.60mm);
\path[line width=0.30mm, draw=L, fill=F] (88.62,52.15) circle (0.50mm);
\path[line width=0.21mm, draw=L] (78.43,58.90) -- (67.99,51.99);
\path[line width=0.21mm, draw=L] (78.27,58.98) -- (88.62,52.31);
\path[line width=0.30mm, draw=L, fill=F] (68.15,41.08) circle (0.50mm);
\path[line width=0.30mm, draw=L, fill=F] (75.22,41.08) circle (0.50mm);
\path[line width=0.30mm, draw=L, fill=F] (88.62,41.16) circle (0.50mm);
\path[line width=0.21mm, draw=L] (68.07,52.07) -- (68.07,40.68);
\path[line width=0.30mm, draw=L] (67.99,41.00);
\path[line width=0.21mm, draw=L] (68.07,41.08) -- (88.62,52.15);
\path[line width=0.21mm, draw=L] (68.07,52.07) -- (75.30,41.08);
\path[line width=0.21mm, draw=L] (67.99,52.23) -- (88.78,41.08);
\path[line width=0.21mm, draw=L] (75.30,41.08) -- (88.70,51.99);
\path[line width=0.21mm, draw=L] (88.62,52.07) -- (88.62,40.84);
\path[line width=0.30mm, draw=L, fill=F] (64.78,67.73) circle (0.50mm);
\path[line width=0.30mm, draw=L, fill=F] (64.78,63.47) circle (0.50mm);
\path[line width=0.30mm, draw=L, fill=F] (64.78,56.57) circle (0.50mm);
\path[line width=0.21mm, draw=L] (64.78,67.65) -- (78.35,58.90);
\path[line width=0.21mm, draw=L] (64.86,63.39) -- (78.43,58.90);
\path[line width=0.21mm, draw=L] (64.86,56.57) -- (78.43,58.98);
\path[line width=0.30mm, draw=L, fill=F] (80.03,72.78) circle (0.15mm);
\path[line width=0.30mm, draw=L, fill=F] (76.82,72.78) circle (0.15mm);
\path[line width=0.30mm, draw=L, fill=F] (81.48,41.08) circle (0.20mm);
\path[line width=0.30mm, draw=L, fill=F] (83.08,41.08) circle (0.20mm);
\path[line width=0.30mm, draw=L, fill=F] (79.87,41.08) circle (0.20mm);
\path[line width=0.30mm, draw=L, fill=F] (92.07,59.78) circle (0.15mm);
\path[line width=0.30mm, draw=L, fill=F] (92.07,61.38) circle (0.15mm);
\path[line width=0.30mm, draw=L, fill=F] (92.07,58.17) circle (0.15mm);
\path[line width=0.30mm, draw=L, fill=F] (64.86,60.02) circle (0.15mm);
\path[line width=0.30mm, draw=L, fill=F] (64.86,58.41) circle (0.15mm);
\path[line width=0.30mm, draw=L, fill=F] (64.86,61.63) circle (0.15mm);
\path[line width=0.30mm, draw=L, fill=F] (138.31,59.05) circle (0.50mm);
\path[line width=0.30mm, draw=L, fill=F] (124.67,72.70) circle (0.50mm);
\path[line width=0.30mm, draw=L, fill=F] (131.09,72.70) circle (0.50mm);
\path[line width=0.30mm, draw=L, fill=F] (151.96,72.70) circle (0.50mm);
\path[line width=0.30mm, draw=L, fill=F] (145.54,72.70) circle (0.50mm);
\path[line width=0.21mm, draw=L] (125.07,72.78) -- (130.61,72.78);
\path[line width=0.21mm, draw=L] (145.94,72.78) -- (151.48,72.78);
\path[line width=0.21mm, draw=L] (124.67,72.78) -- (138.32,58.82);
\path[line width=0.21mm, draw=L] (131.33,72.78) -- (138.07,58.98);
\path[line width=0.21mm, draw=L] (145.30,72.86) -- (138.56,59.14);
\path[line width=0.21mm, draw=L] (151.64,72.30) -- (138.48,59.14);
\path[line width=0.30mm, draw=L, fill=F] (138.40,72.78) circle (0.15mm);
\path[line width=0.30mm, draw=L, fill=F] (152.04,67.73) circle (0.50mm);
\path[line width=0.30mm, draw=L, fill=F] (152.04,63.47) circle (0.50mm);
\path[line width=0.30mm, draw=L, fill=F] (152.04,56.57) circle (0.50mm);
\path[line width=0.30mm, draw=L, fill=F] (161.99,58.58) circle (0.50mm);
\path[line width=0.21mm, draw=L] (138.80,59.22) -- (151.80,67.40);
\path[line width=0.21mm, draw=L] (138.48,59.06) -- (151.96,63.15);
\path[line width=0.21mm, draw=L] (152.36,67.57) -- (161.83,58.74);
\path[line width=0.21mm, draw=L] (152.44,63.31) -- (161.83,58.66);
\path[line width=0.21mm, draw=L] (138.48,59.14) -- (151.64,56.57);
\path[line width=0.21mm, draw=L] (152.28,56.57) -- (161.83,58.58);
\path[line width=0.30mm, draw=L, fill=F] (152.16,46.63) circle (0.50mm);
\path[line width=0.30mm, draw=L, fill=F] (124.75,67.73) circle (0.50mm);
\path[line width=0.30mm, draw=L, fill=F] (124.75,63.47) circle (0.50mm);
\path[line width=0.30mm, draw=L, fill=F] (124.75,56.57) circle (0.50mm);
\path[line width=0.21mm, draw=L] (124.75,67.65) -- (138.32,58.90);
\path[line width=0.21mm, draw=L] (124.83,63.39) -- (138.40,58.90);
\path[line width=0.21mm, draw=L] (124.83,56.57) -- (138.40,58.98);
\path[line width=0.30mm, draw=L, fill=F] (140.00,72.78) circle (0.15mm);
\path[line width=0.30mm, draw=L, fill=F] (136.79,72.78) circle (0.15mm);
\path[line width=0.30mm, draw=L, fill=F] (151.95,52.08) circle (0.15mm);
\path[line width=0.30mm, draw=L, fill=F] (152.04,59.78) circle (0.15mm);
\path[line width=0.30mm, draw=L, fill=F] (152.04,61.38) circle (0.15mm);
\path[line width=0.30mm, draw=L, fill=F] (152.04,58.17) circle (0.15mm);
\path[line width=0.30mm, draw=L, fill=F] (124.83,60.02) circle (0.15mm);
\path[line width=0.30mm, draw=L, fill=F] (124.83,58.41) circle (0.15mm);
\path[line width=0.30mm, draw=L, fill=F] (124.83,61.63) circle (0.15mm);
\path[line width=0.30mm, draw=L, fill=F] (151.95,50.27) circle (0.15mm);
\path[line width=0.30mm, draw=L, fill=F] (151.95,48.66) circle (0.15mm);
\path[line width=0.21mm, draw=L] (138.16,59.13) -- (152.16,46.53);
\path[line width=0.21mm, draw=L] (152.05,46.63) -- (161.88,58.39);
\path[line width=0.30mm, draw=L, fill=F] (152.05,54.43) circle (0.50mm);
\path[line width=0.21mm, draw=L] (138.37,59.03) -- (151.95,54.43);
\path[line width=0.21mm, draw=L] (152.16,54.43) -- (161.83,58.58);
\draw(61.21,67.57) node[anchor=base west]{\fontsize{5.17}{6.21}\selectfont $z_1$};
\draw(61.21,63.19) node[anchor=base west]{\fontsize{5.17}{6.21}\selectfont $z_2$};
\draw(61.21,54.65) node[anchor=base west]{\fontsize{5.17}{6.21}\selectfont $z_{k\!-\!2}$};
\draw(63.88,73.77) node[anchor=base west]{\fontsize{5.17}{6.21}\selectfont $u_1$};
\draw(70.29,73.77) node[anchor=base west]{\fontsize{5.17}{6.21}\selectfont $u_2$};
\draw(83.96,73.77) node[anchor=base west]{\fontsize{5.17}{6.21}\selectfont $u_{2t_1\!-\!1}$};
\draw(92.30,73.56) node[anchor=base west]{\fontsize{5.17}{6.21}\selectfont $u_{2t_1}$};
\draw(92.83,68.00) node[anchor=base west]{\fontsize{5.17}{6.21}\selectfont $v_1$};
\draw(92.62,63.41) node[anchor=base west]{\fontsize{5.17}{6.21}\selectfont $v_2$};
\draw(92.94,55.29) node[anchor=base west]{\fontsize{5.17}{6.21}\selectfont $v_{t_2}$};
\draw(103.30,58.39) node[anchor=base west]{\fontsize{5.17}{6.21}\selectfont $v$};
\draw(66.76,38.94) node[anchor=base west]{\fontsize{5.17}{6.21}\selectfont $w_1$};
\draw(74.46,38.94) node[anchor=base west]{\fontsize{5.17}{6.21}\selectfont $w_2$};
\draw(87.92,38.94) node[anchor=base west]{\fontsize{5.17}{6.21}\selectfont $w_{t_3}$};
\draw(64.73,51.76) node[anchor=base west]{\fontsize{5.17}{6.21}\selectfont $w_1^{\prime}$};
\draw(89.63,51.66) node[anchor=base west]{\fontsize{5.17}{6.21}\selectfont $w_2^{\prime}$};
\draw(121.18,67.57) node[anchor=base west]{\fontsize{5.17}{6.21}\selectfont $z_1$};
\draw(121.18,63.19) node[anchor=base west]{\fontsize{5.17}{6.21}\selectfont $z_2$};
\draw(121.18,54.65) node[anchor=base west]{\fontsize{5.17}{6.21}\selectfont $z_{k\!-\!2}$};
\draw(123.85,73.77) node[anchor=base west]{\fontsize{5.17}{6.21}\selectfont $u_1$};
\draw(130.26,73.77) node[anchor=base west]{\fontsize{5.17}{6.21}\selectfont $u_2$};
\draw(143.93,73.77) node[anchor=base west]{\fontsize{5.17}{6.21}\selectfont $u_{2t_1\!-\!1}$};
\draw(92.30,73.56) node[anchor=base west]{\fontsize{5.17}{6.21}\selectfont $u_{2t_1}$};
\draw(92.30,73.56) node[anchor=base west]{\fontsize{5.17}{6.21}\selectfont $u_{2t_1}$};
\draw(152.37,73.56) node[anchor=base west]{\fontsize{5.17}{6.21}\selectfont $u_{2t_1}$};
\draw(152.80,68.00) node[anchor=base west]{\fontsize{5.17}{6.21}\selectfont $v_1$};
\draw(152.59,63.41) node[anchor=base west]{\fontsize{5.17}{6.21}\selectfont $v_2$};
\draw(163.27,58.39) node[anchor=base west]{\fontsize{5.17}{6.21}\selectfont $v$};
\draw(77.13,56.68) node[anchor=base west]{\fontsize{5.17}{6.21}\selectfont $u^*$};
\draw(136.35,56.78) node[anchor=base west]{\fontsize{5.17}{6.21}\selectfont $u^*$};
\draw(152.69,57.74) node[anchor=base west]{\fontsize{5.17}{6.21}\selectfont $v_{t_2}$};
\draw(152.59,52.94) node[anchor=base west]{\fontsize{5.17}{6.21}\selectfont $w_1^{\prime}$};
\draw(152.59,44.82) node[anchor=base west]{\fontsize{5.17}{6.21}\selectfont $w_{t_3}$};
\end{tikzpicture}%
\begin{center}\small{Fig. 1 : The vertex labels of $K_1(t_1, t_2, t_3)$ and $K_1(t_1, t_2+t_3+2, 0).$
}\end{center}
\end{figure}

\begin{pro}\label{pro::3.1}
$\rho(K_1(t_1, t_2, t_3))<\rho(K_1(t_1, t_2+t_3+2, 0))$, where $2t_1+t_2+t_3+k+2=n$.
\end{pro}
\begin{proof}
Suppose that $E_1=\{vw_i'\mid 1\leq i\leq 2\}+\{u^*w_i\mid 1\leq i\leq t_3\}+\{vw_i\mid 1\leq i\leq t_3\}$ and $E_2=\{w_i'w_j\mid 1\leq i\leq 2, 1\leq j\leq t_3\}$. Let $G_1=K_1(t_1, t_2, t_3)$ and let $G_1'=G_1-E_2+E_1$, where the labels of $V(G)$ and $V(G')$ are illustrated in Fig. 1. It is evident that $G_1'\cong K_1(t_1, t_2+t_3+2, 0)$. Let 
$x$ be the Perron vector of $A(G_1)$ and  $\rho'=\rho(G_1)$. By symmetry, we have $x_{z_i}=x_{z_1}$ for $2\leq i\leq k-2$, $x_{u_i}=x_{u_1}$ for $2\leq i\leq 2t_1$, $x_{v_i}=x_{v_1}$ for $2\leq i\leq t_2$, $x_{w_1'}=x_{w_2'}$ and  $x_{w_i}=x_{w_1}$ for $2\leq i\leq t_3$. Note that $K_{2,{t_3+1}}$ is a proper subgraph of $G_1$. Then $\rho'> \sqrt{2(t_3+1)}$ by Lemma \ref{lem::2.2}. And then
$$
\left\{\begin{array} { l }
{ \rho'x_{w_{1}}=2x_{w_1'}, } \\
{ \rho'x_{w_1'}=x_{u^*}+t_3x_{w_1}, }
\end{array} \quad \left\{\begin{array}{l}
\rho'x_{v}=t_2x_{v_1},  \\
\rho'x_{v_1}=x_{u^*}+x_{v}.
\end{array}\right.\right.
$$
From this information, we can induce that
\begin{equation}\label{equ::1}
\begin{aligned}
x_{u^*}=({\rho'}^2-2t_3)x_{w_1'}\text{ and } x_v=\frac{t_2}{{\rho'}^2-t_2}x_{u^*}.
\end{aligned}
\end{equation}
From (\ref{equ::1}), it follows that $x_{u^*}+x_v=\frac{{\rho'}^2}{{\rho'}^2-t_2}x_{u^*}=\frac{{\rho'}^2({\rho'}^2-2t_3)}{{\rho'}^2-t_2}x_{w_1'}>2x_{w_1'}$ because $\rho'> \sqrt{2(t_3+1)}$.
Furthermore, combining this with 
$x_{u_2}>x_u$ and $2x_{v_{n-3}}>x_{v_\delta}$, we conclude that
\begin{align*}
\rho(G_1')-\rho(G_1)&\geq x^T\big(A(G_1')-A(G_1)\big)x\\
&\geq4x_{v}x_{w_1'}+2\sum_{i=1}^{t_3}x_{u^*}x_{w_i}+2\sum_{i=1}^{t_3}x_{v}x_{w_i}-4\sum_{i=1}^{t_3}x_{w_1'}x_{w_i}\\
&=4x_{v}x_{w_1'}+2t_3(x_{u^*}+x_{v})x_{w_1}-4t_3x_{w_1'}x_{w_1}\\
&>2t_3(x_{u^*}+x_{v}-2x_{w_1'})x_{w_1}\\
&>0 ~~ (\text{since}~ x_{u^*}+x_{v}-2x_{w_1'}>0).
\end{align*}
It follows that $\rho(G_1')>\rho(G_1)$.
\end{proof}

\begin{figure}[http]
\centering
\begin{tikzpicture}[x=1.00mm, y=1.00mm, inner xsep=0pt, inner ysep=0pt, outer xsep=0pt, outer ysep=0pt]
\path[line width=0mm] (61.58,49.07) rectangle +(108.63,29.00);
\definecolor{L}{rgb}{0,0,0}
\definecolor{F}{rgb}{0,0,0}
\path[line width=0.30mm, draw=L, fill=F] (80.72,59.05) circle (0.50mm);
\path[line width=0.30mm, draw=L, fill=F] (67.08,72.70) circle (0.50mm);
\path[line width=0.30mm, draw=L, fill=F] (73.50,72.70) circle (0.50mm);
\path[line width=0.30mm, draw=L, fill=F] (94.37,72.70) circle (0.50mm);
\path[line width=0.30mm, draw=L, fill=F] (87.95,72.70) circle (0.50mm);
\path[line width=0.21mm, draw=L] (67.48,72.78) -- (73.02,72.78);
\path[line width=0.21mm, draw=L] (88.35,72.78) -- (93.89,72.78);
\path[line width=0.21mm, draw=L] (67.08,72.78) -- (80.72,58.82);
\path[line width=0.21mm, draw=L] (73.74,72.78) -- (80.48,58.98);
\path[line width=0.21mm, draw=L] (87.71,72.86) -- (80.96,59.14);
\path[line width=0.21mm, draw=L] (94.05,72.30) -- (80.88,59.14);
\path[line width=0.30mm, draw=L, fill=F] (80.80,72.78) circle (0.15mm);
\path[line width=0.30mm, draw=L, fill=F] (94.45,67.73) circle (0.50mm);
\path[line width=0.30mm, draw=L, fill=F] (94.45,63.47) circle (0.50mm);
\path[line width=0.30mm, draw=L, fill=F] (94.45,56.57) circle (0.50mm);
\path[line width=0.30mm, draw=L, fill=F] (104.40,58.58) circle (0.50mm);
\path[line width=0.21mm, draw=L] (81.20,59.22) -- (94.21,67.40);
\path[line width=0.21mm, draw=L] (80.88,59.06) -- (94.37,63.15);
\path[line width=0.21mm, draw=L] (94.77,67.57) -- (104.24,58.74);
\path[line width=0.21mm, draw=L] (94.85,63.31) -- (104.24,58.66);
\path[line width=0.21mm, draw=L] (80.89,59.14) -- (94.05,56.57);
\path[line width=0.21mm, draw=L] (94.69,56.57) -- (104.24,58.58);
\path[line width=0.30mm, draw=L, fill=F] (67.16,67.73) circle (0.50mm);
\path[line width=0.30mm, draw=L, fill=F] (67.16,63.47) circle (0.50mm);
\path[line width=0.30mm, draw=L, fill=F] (67.16,56.57) circle (0.50mm);
\path[line width=0.21mm, draw=L] (67.16,67.65) -- (80.72,58.90);
\path[line width=0.21mm, draw=L] (67.24,63.39) -- (80.80,58.90);
\path[line width=0.21mm, draw=L] (67.24,56.57) -- (80.80,58.98);
\path[line width=0.30mm, draw=L, fill=F] (82.41,72.78) circle (0.15mm);
\path[line width=0.30mm, draw=L, fill=F] (79.20,72.78) circle (0.15mm);
\path[line width=0.30mm, draw=L, fill=F] (94.45,59.78) circle (0.15mm);
\path[line width=0.30mm, draw=L, fill=F] (94.45,61.38) circle (0.15mm);
\path[line width=0.30mm, draw=L, fill=F] (94.45,58.17) circle (0.15mm);
\path[line width=0.30mm, draw=L, fill=F] (67.24,60.02) circle (0.15mm);
\path[line width=0.30mm, draw=L, fill=F] (67.24,58.41) circle (0.15mm);
\path[line width=0.30mm, draw=L, fill=F] (67.24,61.63) circle (0.15mm);
\draw(63.58,67.57) node[anchor=base west]{\fontsize{5.17}{6.21}\selectfont $z_1$};
\draw(63.58,63.19) node[anchor=base west]{\fontsize{5.17}{6.21}\selectfont $z_2$};
\draw(63.58,54.65) node[anchor=base west]{\fontsize{5.17}{6.21}\selectfont $z_{k\!-\!2}$};
\draw(66.26,73.77) node[anchor=base west]{\fontsize{5.17}{6.21}\selectfont $u_1$};
\draw(72.67,73.77) node[anchor=base west]{\fontsize{5.17}{6.21}\selectfont $u_2$};
\draw(86.34,73.77) node[anchor=base west]{\fontsize{5.17}{6.21}\selectfont $u_{2t_1\!-\!1}$};
\draw(94.78,73.56) node[anchor=base west]{\fontsize{5.17}{6.21}\selectfont $u_{2t_1}$};
\draw(95.21,68.00) node[anchor=base west]{\fontsize{5.17}{6.21}\selectfont $v_1$};
\draw(94.99,63.41) node[anchor=base west]{\fontsize{5.17}{6.21}\selectfont $v_2$};
\draw(105.68,58.39) node[anchor=base west]{\fontsize{5.17}{6.21}\selectfont $v$};
\draw(78.75,56.78) node[anchor=base west]{\fontsize{5.17}{6.21}\selectfont $u^*$};
\draw(95.10,57.74) node[anchor=base west]{\fontsize{5.17}{6.21}\selectfont $v_{t_2}$};
\path[line width=0.30mm, draw=L, fill=F] (140.53,63.48) circle (0.50mm);
\path[line width=0.30mm, draw=L, fill=F] (154.26,74.37) circle (0.50mm);
\path[line width=0.30mm, draw=L, fill=F] (154.26,70.12) circle (0.50mm);
\path[line width=0.30mm, draw=L, fill=F] (154.26,63.21) circle (0.50mm);
\path[line width=0.30mm, draw=L, fill=F] (164.21,65.22) circle (0.50mm);
\path[line width=0.21mm, draw=L] (141.01,63.65) -- (154.38,74.35);
\path[line width=0.21mm, draw=L] (140.69,63.49) -- (154.27,70.14);
\path[line width=0.21mm, draw=L] (154.58,74.21) -- (164.05,65.38);
\path[line width=0.21mm, draw=L] (154.66,69.96) -- (164.05,65.30);
\path[line width=0.21mm, draw=L] (140.70,63.57) -- (154.04,63.16);
\path[line width=0.21mm, draw=L] (154.50,63.21) -- (164.05,65.22);
\path[line width=0.30mm, draw=L, fill=F] (154.37,53.28) circle (0.50mm);
\path[line width=0.30mm, draw=L, fill=F] (126.97,72.16) circle (0.50mm);
\path[line width=0.30mm, draw=L, fill=F] (126.97,67.90) circle (0.50mm);
\path[line width=0.30mm, draw=L, fill=F] (126.97,61.00) circle (0.50mm);
\path[line width=0.21mm, draw=L] (126.97,72.08) -- (140.53,63.33);
\path[line width=0.21mm, draw=L] (127.05,67.82) -- (140.61,63.33);
\path[line width=0.21mm, draw=L] (127.05,61.00) -- (140.61,63.41);
\path[line width=0.30mm, draw=L, fill=F] (154.16,58.73) circle (0.15mm);
\path[line width=0.30mm, draw=L, fill=F] (154.26,66.42) circle (0.15mm);
\path[line width=0.30mm, draw=L, fill=F] (154.59,70.02) circle (0.15mm);
\path[line width=0.30mm, draw=L, fill=F] (154.26,64.82) circle (0.15mm);
\path[line width=0.30mm, draw=L, fill=F] (127.05,64.45) circle (0.15mm);
\path[line width=0.30mm, draw=L, fill=F] (127.05,62.85) circle (0.15mm);
\path[line width=0.30mm, draw=L, fill=F] (127.05,66.06) circle (0.15mm);
\path[line width=0.30mm, draw=L, fill=F] (154.16,56.91) circle (0.15mm);
\path[line width=0.30mm, draw=L, fill=F] (154.16,55.31) circle (0.15mm);
\path[line width=0.21mm, draw=L] (140.38,63.56) -- (154.27,53.52);
\path[line width=0.21mm, draw=L] (154.27,53.28) -- (164.10,65.03);
\path[line width=0.30mm, draw=L, fill=F] (154.27,61.08) circle (0.50mm);
\path[line width=0.21mm, draw=L] (140.59,63.46) -- (154.16,61.05);
\path[line width=0.21mm, draw=L] (154.37,61.08) -- (164.05,65.22);
\draw(123.39,72.00) node[anchor=base west]{\fontsize{5.17}{6.21}\selectfont $z_1$};
\draw(123.39,67.62) node[anchor=base west]{\fontsize{5.17}{6.21}\selectfont $z_2$};
\draw(123.39,59.08) node[anchor=base west]{\fontsize{5.17}{6.21}\selectfont $z_{k\!-\!2}$};
\draw(155.02,74.65) node[anchor=base west]{\fontsize{5.17}{6.21}\selectfont $v_1$};
\draw(154.80,70.05) node[anchor=base west]{\fontsize{5.17}{6.21}\selectfont $v_2$};
\draw(165.49,65.03) node[anchor=base west]{\fontsize{5.17}{6.21}\selectfont $v$};
\draw(138.56,61.21) node[anchor=base west]{\fontsize{5.17}{6.21}\selectfont $u^*$};
\draw(154.91,64.39) node[anchor=base west]{\fontsize{5.17}{6.21}\selectfont $v_{t_2}$};
\draw(154.80,59.58) node[anchor=base west]{\fontsize{5.17}{6.21}\selectfont $u_1$};
\draw(154.80,51.46) node[anchor=base west]{\fontsize{5.17}{6.21}\selectfont $u_{2t_1}$};
\end{tikzpicture}%
\begin{center}\small{Fig. 2 : The vertex labels of $K(t_1, t_2)$ and $K(0, 2t_1+t_2)$.
}\end{center}
\end{figure}

\begin{pro}\label{pro::3.2}
$\rho(K(t_1, t_2))<\rho(K(0, 2t_1+t_2))$, where $2t_1+t_2+k=n$ and $t_2\neq 0$.
\end{pro}
\begin{proof}
Assume that $E_1=\{vu_i\mid 1\leq i\leq 2t_1\}$ and $E_2=\{u_iu_{i+1}\mid 1\leq i\leq 2t_1-1\text{ and } i \text{ is an odd integer}\}$. Let $G_2=K(t_1, t_2)$ and $G_2'=G_2-E_2+E_1$.  Clearly, $\rho(G_2')\cong\rho(K(0, 2t_1+t_2))$. Their vertices be labeled as in Fig. 2. It suffices to show $\rho(G_2)<\rho(G_2')$. 

Suppose to the contrary that $\rho(G_2)\geq \rho(G_2')$. First of all, for $k\geq 3$, it follows that
$e(G_2)=3t_1+2t_2+k-2$ and $\delta =1$ for $k\geq 3$.
Combining this with Lemma \ref{lem::2.5}, we have
\begin{align*}
\rho(G_2)& \leq \sqrt{2(3t_1+2t_2+k-2)-n+1}\\
& \leq \sqrt{6t_1+4t_2+2k-n-3}\\
&<\sqrt{4n-n}~~ (\text{since}~n=2t_1+t_2+k)\\
&=\sqrt{3n}.
\end{align*}
Furthermore, a simple calculation shows that $\rho(G_2')\leq \rho(G_2)<\sqrt{3n}<n-k$ for $k\geq3$ when $n\geq 4k$.
Let 
$x$ be the Perron vector of $A(G_2)$, and $y$ be the Perron vector of $A(G_2')$. By symmetry, we have $x_{z_i}=x_{z_1}$ for $2\leq i\leq k-2$, $x_{u_i}=x_{u_1}$ for $2\leq i\leq 2t_1$, $x_{v_i}=x_{v_1}$ for $2\leq i\leq t_2$, $y_{z_i}=y_{z_1}$ for $2\leq i\leq k-2$ and $y_{u_i}=y_{v_j}=y_{u_1}$ for $2\leq i\leq t_2$ and $1\leq j\leq t_2$. Notice that
$\rho(G_2')y_v=(n-k)y_{u_1}$. Then $y_v> y_{u_1}$ since $\rho(G_2')< n-k$. And then
\begin{align*}
y^T(\rho(G_2')-\rho(G_2))x& =y^T(A(G_2')-A(G_2))x \\
&=\sum_{vu_i\in E_1}(x_vy_{u_i}+x_{u_i}y_v)-\sum_{u_iu_{i+1}\in E_2}(x_{u_i}y_{u_{i+1}}+x_{u_{i+1}}y_{u_i})\\
&=2t_1(x_{u_1}y_v+x_vy_{u_1})-2t_1x_{u_1}y_{u_1}\\
&=2t_1(x_{u_1}(y_v-y_{u_1})+x_vy_{u_1})\\
&>0 ~~ (\text{since}\  y_v> y_{u_1}),
\end{align*}
a contradiction. It follows that $\rho(G_2')>\rho(G_2)$.
\end{proof}

\renewcommand\proofname{\bf Proof of Theorem \ref{thm::1}}
\begin{proof}

Suppose that $G^*$ is a graph that attains the maximum spectral radius among all minimally $(k, k)$-edge-connected graphs of order $n$. Let $x$ be the Perron vector of $A(G^*)$ with coordinate $x_{u^*}=\max\{x_i \mid i\in V(G^*)\}$. Denote by $B=\{v\mid d_{G^*}(v)=1\}$, $A=N(u^*)\setminus B$ and $R=V(G^*)\setminus N[u^*]$. According to $(i)$ of Lemma \ref{lem::2.7}, we have $0\leq|B|\leq k-2$. The graph $G^*$ is the intersection of some internal paths. The following assertions present the structural properties of the extremal graph $G^*$. 

{\flushleft\bf Claim 1.} $u^*$ is the unique cut vertex of $G^*$.
\renewcommand\proofname{\bf Proof}
\begin{proof}
If not, $G^*$ either contains no cut vertices or at least two cut vertices. If $G^*$ contains no cut vertices, then $|B|=0$. Let $p$ denote the number of internal paths. By Corollary \ref{cor::1}, the length of each internal path must be at least $k$. If $p=0$, then $G^*\cong C_n$. Let $C_n:=u^*u_2\cdots u_nu^*$ and define $G_1=G^*-\{u_2u_3\}+\{u^*u_3\}.$ 
By symmetry, $x_{u_i}=u^*$ for $2\leq i\leq n$. Moreover, $G_1$ remains minimally $(k,k)$-edge-connected for $n\geq 4k$. However, Lemma \ref{lem::2.4} implies $\rho(G_1)>\rho(G^*)$, a contradiction. If $p\geq 1$, consider an internal path $P_1:=v_1u_1u_2\cdots u_tv_2$ in $G^*$, where $v_1\neq u^*$ or $v_2\neq u^*$. Define 
$$G_2=G^*-\{u_1v_1,u_tv_2\}+\{u_1u^*,u_tu^*\}.$$ 
By Corollary \ref{cor::1}, $t+1\geq k$ ensures $G_2$ stays minimally $(k,k)$-edge-connected. Therefore, it follows that $\rho(G_2)>\rho(G^*)$ by Lemma \ref{lem::2.4}, a contradiction. We therefore conclude that $G^*$ must contain cut vertices. If $u^*$ is not a cut vertex of $G^*$, then there exists another cut vertex $u$ of $G^*$ such that $u^*\in C$, where $C$ is a component of $G^*-u$. For any $v\in N(u)\setminus N(u^*)$ with $v\notin V(C)$, let $G_3=G^*-\{uv\}+\{u^*v\}$, which is also minimally $(k,k)$-edge-connected. Consequently, $\rho(G_3)>\rho(G^*)$ by $x_{u^*}=\max\{x_i \mid i\in V(G^*)\}$ and Lemma \ref{lem::2.4} again, a contradiction. Thus, $u^*$ must be the unique cut vertex of $G^*$. Now, suppose there exists another cut vertex $u$. An analogous argument yields a minimally $(k,k)$-edge-connected graph whose spectral radius exceeds $\rho(G^*)$, a contradiction.
\end{proof}

{\flushleft\bf Claim 2.} $G^*[A]$ is isomorphic to the union of some independent edges and isolated vertices. 
\begin{proof}
On the one hand, $G^*[A]$ contains no cycles. Otherwise, suppose there exists a cycle $C_l\subset G^*[A]$ $(l\geq 3)$. Then there exists a wheel $W_{l+1}$ in $G'$, which would form a chorded cycle in $G^*$, contradicting with Corollary \ref{cor::1}. On the other hand, if $G^*[A]$ contains a path $P_3$, then it must also contain a chorded cycle of order 4, a contradiction.
\end{proof}
Let $A_1$ be the isolated vertex set of $G^*[A]$. Then $A_2=A\setminus A_1$ consists of some independent edges if $A_2\neq\emptyset$.

{\flushleft\bf Claim 3.} $N_R(u)=\emptyset$ for any $u \in A_2$.
\begin{proof}
Otherwise, there exists a vertex $v \in R$ and $u_2 \in A_2$ such that $vu_2\in E(G^*)$. We may further assume that $u_2^{\prime} \in N_{G^*}(u_2) \cap A_2$. If $N_{R}(v)=\emptyset$, then there exists a vertex $u \in A$ such that $uv\in E(G^*)$ due to $\delta(G^*) \geq 2$. It follows that
$$
\begin{cases}C=u^* u_2^{\prime} u_2 v u_1u^* \text { is a cycle with the chord } u^* u_2 & \text { if } u=u_1 \in A_1; \\ 
C=u^* u_2^{\prime} v u_2u^* \text { is a cycle with the chord } u_2^{\prime} u_2 & \text { if } u=u_2^{\prime} \in A_2; \\ 
C=u^* u_2^{\prime} u_2 v u_3u^* \text { is a cycle with the chord } u_2^* u_2 & \text { if } u=u_3 \in A_2.
\end{cases}
$$
It is impossible because any cycle of $G^*$ has no chord by Corollary \ref{cor::1}. So, $d_R(v) \geq 1$. However, in this situation, we assert that there exists a path $P:=v v_1 \cdots v_t$ in $G^*[R]$ such that $v_t$ is adjacent to some $u_i^{\prime} \in A$. Otherwise, $u_2 v$ will be a cut edge, a contradiction. By regarding $u$ as the above $u^*$, as similar above we can find a chorded cycle in $G^*$, a contradiction.
\end{proof}
If $A_2 \neq \emptyset$, then, from Claims 1-3, $\{u^*\} \cup A_2$ induces $t_1=\frac{|A_2|}{2}$'s triangles with a common vertex $u^*$. Moreover, Claim 2 shows that each of these triangles must be a leaf block of $G^*$. Now, we consider two cases:

{\flushleft\bf{Case 1.}} $|B|\leq k-3$.


{\flushleft\bf Claim 4.} $|B|= k-3$.
\begin{proof}
If not, then $|B|=t\leq k-4$. Note that $u^*v\in E(G^*)$ for any $v\in B$ by Claim 1. Moreover, $G^*\setminus B$ is minimally $(k-t, k-t)$-edge-connected by Lemma \ref{lem::2.7}. First, we show that $G^*\setminus B$ contains no internal paths. Suppose otherwise, and let $P_1:=u^*u_1u_2\cdots u_p$ be an internal path in $G^*\setminus B$. By Corollary \ref{cor::1}, we have $p\geq k-t$. Construct $G_4=G^*-\{u_1u_2\}+\{u_2u^*\}$. Then Lemma \ref{lem::2.4} implies $\rho(G_4)>\rho(G^*)$. It is a contradiction since $G_4$ is also minimally $(k,k)$-edge-connected and $|B|=t+1\leq k-3$. So there is no internal path in $G^*\setminus B$, and then $G^*\setminus B$ must consist of cycles intersecting only at $u^*$. Using similar reasoning, we can construct a minimally $(k,k)$-edge-connected graph with greater spectral radius than $\rho(G^*)$, again yielding a contradiction.  We conclude that $|B|= k-3$.
\end{proof}

{\flushleft\bf Claim 5.} $G^*[A_1\cup R]$ is isomorphic to the union of some path. 
\begin{proof}
Equivalently, we need to show $d_{G^*}(v)=2$ for any $v\in A_1\cup R$. If not, there exists $v_1\in A_1\cup R$ with $d_{G^*}(v_1)=r\geq 3$. Suppose $v_1$ belongs to a leaf block $B_1$. Let $v_1u_i'\in E(G^*)$ for $1\leq i\leq r$. This yields that $u_i'$ must be a  vertex of the internal path $P_i$ in $B_1$ for any $1\leq i\leq r$. Moreover, $|P_i|=t+1\geq4$ holds for each path. Consider $P_1:=v_1u_1'u_2\cdots u_{t-1}v_{2}$ and let $G_5=G^*-\{v_1u_1', u_{t-1}v_2\}+\{u_1'u^*, u_{t-1}u^*\}$. The graph $G_5$ remains minimally $(k,k)$-edge-connected. Combining this with Lemma \ref{lem::2.4}, we obtain $\rho(G_5)>\rho(G^*)$, which contradicts the maximality of $\rho(G^*)$. Thus, $d_{G^*}(v)=2$ for any $v\in A_1\cup R$.
\end{proof} 
Thus, $G^*[u^*\cup A_1\cup R]$ is isomorphic to some cycles that have a common vertex $u^*$. 

{\flushleft\bf Claim 6.} There are no cycles with a length greater than 4 in $G^*$.
\begin{proof}
If not, there exists a cycle $C_l:=u^*u_2\cdots u_lu^*$ with $l\geq 5$. Let $G_6=G^*-\{u_2u_3, u_{l-1}u_l\}+\{u^*u_3, u^*u_{l-1}, u_2u_l\}$. Then $G_6$ is also a minimally $(k,k)$-edge-connected graph. Combining with Lemma \ref{lem::2.4} again, we have $\rho(G_6)>\rho(G^*)$, a contradiction.
\end{proof}

{\flushleft\bf Claim 7.} There is at most one cycle of length 4 in $G^*$.
\begin{proof}
Otherwise, there exists two cycles $C_1:=u^*v_1v_2v_3u^*$ and $C_2:=u^*w_1w_2w_3u^*$ in $G^*$. Let $G_7=G^*-\{v_1v_2,w_1w_2,v_2v_3,w_2w_3\}+\{v_1v_3,w_1w_3,v_2u^*, w_2u^*,v_2w_2\}$. Then $G_7$ is also minimally $(k,k)$-edge-connected and $\rho(G_7)>\rho(G^*)$ by Lemmas \ref{lem::2.4} and \ref{lem::2.8}, a contradiction.
\end{proof}

Consider a $4$-cycle $C:=u^*u_1u_2u_3u^*$ in $G^*$. Let $G_8=G^*-\{u_1u_2\}+\{u^*u_2\}$. Then $G_8$ remains minimally $(k,k)$-edge-connected and  $\rho(G_8)>\rho(G^*)$ by Lemma \ref{lem::2.4}, a contradiction. If no $4$-cycle exists in $G^*$, then $G^*\cong K_2(\frac{n-k+2}{2},0,0)$. By Lemmas \ref{lem::3.1} and \ref{lem::3.2}, we have $\rho(K_2(\frac{n-k+2}{2},0,0))<\rho(K^{k-2}_{2, n-k})$. This again leads to a contradiction since $K^{k-2}_{2, n-k}$ is minimally $(k,k)$-edge-connected.

{\flushleft\bf{Case 2.}} $|B|= k-2$. 

In this case, Lemma \ref{lem::2.7}$(iv)$ implies that $G^*\setminus B$ is minimally $(2,2)$-edge-connected. Since any block of a minimally $(2,2)$-edge-connected graph is itself minimally $(2,2)$-edge-connected, let $B_1, B_2, \ldots, B_t$ be the $t$ blocks of $G^*\setminus B$. Recall that $A=N(u^*)\setminus B$, $R=V(G^*)\setminus N[u^*]$ and $A_1$ is the isolated vertex set of $G^*[A]$. Let $V(B_l)\cap A_1=\{v_{l1}, v_{l2}, \ldots, v_{lb_{l}}\}$ for $1\leq l\leq t$.

{\flushleft\bf Claim 8.} $E(R)=\emptyset$.
\begin{proof}
Suppose there exists $b_{l}\geq 3$ for some $1\leq l\leq t$, without loss of generality, assume $b_{1}\geq 3$. 
We first prove $d_{G^*}(v)=2$ for all $v\in A_1$. If not, there exists $p\leq b_{1}$ with $d_{G^*}(v_{1p})=r\geq3$. This would imply $\kappa'_2(G^*-u^*v_{1p})\geq 2$, but Lemma \ref{lem::2.6} shows $\kappa'_2(G^*-u^*v_{1p})=1$, a contradiction. Thus, $d_{G^*}(v)=2$ hold for all $v\in A_1$.  
Now consider any $v_{1i}\in A_1$ with $|N_R(v_{1i})\cap N_R(v_{1j})|=\emptyset$ where $i\neq j$. Let $v_{1i}w\in E(G^*)$ and let $$G_9=G^*-\{wv\mid v\in N(w)\cap R\}+\{u^*w\}+\{u^*v\mid v\in N(w)\cap R\}.$$ Then $G_9$ remins minimally $(k,k)$-edge-connected, and Lemma \ref{lem::2.4} yields $\rho(G_9)>\rho(G^*)$, a contradiction.
For any $v_{1i}\in A_1$ with $|N_R(v_{1i})\cap N_R(v_{1j})|\neq\emptyset$ where $i\neq j$, let $N_R(v_{1i})=u$. Then $G^*[N_{A_1}[u]]\cong K_{2,N_{A_1}(u)}$.
Suppose $N_R(A_1\cap B_1)=\{w_1, w_2,\ldots, w_q\}$. 
If $q\geq 3$, we claim $N_R(w_i)\cap N_R(w_j)\cap N_R(w_k)=\emptyset$ where $i\neq j\neq k$. 
Otherwise, let $N_R(v_{1i})\cap N_R(v_{1j})\cap N_R(v_{1k})=w$ and consider $G^*-v_{1i}w$. Since $\kappa'_2(G^*-v_{1i}w)\geq 2$ would hold, this contradicts Lemma \ref{lem::2.6} which gives $\kappa'_2(G^*-v_{1i}w)=1$. 
We now prove that $|N_R(A_1\cap B_1)|\leq 2$. Suppose otherwise that $|N_R(A_1\cap B_1)|\geq 3$ and let $\{w_1,w_2,w_3\}\subset N_R(A_1\cap B_1)$. Since $w_1$ and $w_2$ belong to the same block $B_1$, there exists a path $P$ connecting them in $R$. Moreover, as ${w_1,w_2,w_3}\subseteq B_1$, there must be exactly one path $P_1$ between $w_3$ and some vertex $v$ on $P$.
We further claim that $d_{G^*}(w)=2$ for any $w$ with $wv\in E(P_1)$. If not, we would have $\kappa'_2(G^*-wv)\geq 2$. This again leads to a contradiction when combined with Lemma \ref{lem::2.6}, as we find $\kappa'_2(G^*-wv)=1$. Consider a graph $G'=G^*-wv+wu^*$. Then $G'$ is also a minimally $(k,k)$-edge-connected graph. Combining with Lemma \ref{lem::2.4} again, we have $\rho(G')>\rho(G^*)$, a contradiction.
Thus, we have $|N_R(A_1\cap B_1)|\leq 2$. If $|N_R(A_1\cap B_1)|=1$, then $B_1\cong K_{2,b_1}$. 
Assume $|N_R(A_1\cap B_1)|=2$ with $w_1, w_2\in N_R(A_1)$. We assert that $d_{R}(w_1)=d_{R}(w_2)=1$. If $d_{R}(w_1)\geq 2$ for some $w_1w'\in E(G^*)$, then $\kappa'_{2}(B_1-w_1w')=2$ since $B_1-w_1w'$ contains a cycle, a contradiction.
When $N_R(w_1)\cap N_R(w_2)=\emptyset$, let $w_1w'_1\in E(G^*)$ and $w_2w'_2\in E(G^*)$, and define $$G_{10}=G^*-w_1w'_1-w_2w'_2+u^*w'_1+u^*w'_2.$$ Then $G_{10}$ remains minimally $(k,k)$-edge-connected graph. Combining with Lemma \ref{lem::2.4} again, we have $\rho(G_{10})>\rho(G^*)$, a contradiction. For $N_R(w_1)\cap N_R(w_2)=\{w^*\}$, we prove $N_{G^*}(w^*)=\{w_1, w_2\}$. If not, $\kappa'_2(G^*-w^*w_1)\geq 2$ would hold. Combining with Lemma \ref{lem::2.6} again, which contradicts with $\kappa'_2(G^*-w^*w_1)=1$. Now, consider $G^*-w_1w^*+u^*w^*$. We have $B_1\cong K_{2,i_1}$ and $\rho(G^*-w_1w^*+u^*w^*)>\rho(G^*)$ by Lemma \ref{lem::2.4} again, a contradiction. We conclude that when $b_{l}\geq 3$ exists, $E(R)=\emptyset$ must hold.

Suppose there exists $b_{l}=2$ where $1\leq l\leq t$, without lose of generality, assume that $b_{1}= 2$. Then $B_1\cap A_1=\{v_{11}, v_{12}\}$. 
Suppose that $N_R(v_{11})\cap N_R(v_{12})=I=\{w_1, w_2,\ldots, w_q\}$. If $q\emph{}\geq 1$, then we assert that $N_{R\setminus I}(v_{1i})=\emptyset$ for $1\leq i\leq 2$. Otherwise, there is a vertex $w\in R\setminus I$ that is adjacent to a vertex $v_{11}\in A_1$. Then $d_{G^*}(w_i)=2$ for $1\leq i\leq q$. If not, $\kappa'_2(B_1-v_{11}w_i)=2$ since $w_i$ is not a cut vertex, a contradiction. Assume that a path $P:=ww'_1\cdots w'_p\subseteq R$. If $w'_p=w \in R$, then $v_{11}w$ is a cut edge, which contradicts that $B_1$ is a block. If $w'_p \in N_R(v_{11})$, then $v_{11}$ is a cut vertex, which also contradicts that $B_1$ is a block. If $w'_p \in N_R(v_{12})$, then let $$G_{11}=G^*-v_{11}w-v_{12}w'_p+u^*w+u^*w'_p.$$ It is evident that $G_{11}$ is also a minimally $(k,k)$-edge-connected graph and $\rho(G_{11})>\rho(G^*)$, a contradiction. 
If $q=0$, then let $G_{12}=G^*-\{v_{11}v\mid v\in N_R(v_{11})\}-\{v_{12}v\mid v\in N_R(v_{12})\}+\{u^*v\mid v\in N_R(v_{11})\}+\{u^*v\mid v\in N_R(v_{11})\}+\{v_{11}v_{12}\}$. Obviously, $G_{12}$ is also a minimally $(k,k)$-edge-connected graph and $\rho(G_{12})>\rho(G^*)$ by Lemma \ref{lem::2.4}, a contradiction. Thus, in the case where exists $b_{l}= 2$, we also have $E(R)=\emptyset$.
\end{proof}
By Claim 3, every vertex in $R$ is adjacent only to some vertices in $A_1$, that is, $N_{A}(w_i)=N_{A_1}(w_i)$ for any $w_i \in R$.

{\flushleft\bf Claim 9.} $G^*\cong K(0, n-k)$ for $|V(R)| =1$.
\begin{proof}
We may assume that $R=\{w_1\}$, which leads to $G^*[N_{A_1}(w_1) \cup\{w_1, u^*\}] \cong K_{2, a_1}$, where $a_1=|N_{A_1}(w_1)|$. Consequently,
$G^* \cong K(t_1, a_1)$ for some positive $t_1 \geq 0, a_1 \geq 2$ satisfying $2 t_1+ a_1=n-k$ by Claim 8. Proposition \ref{pro::3.2} then gives $\rho(G^*)=\rho(K(t_1, a_1))<\rho(K(0, 2t_1+a_1))=\rho(K(0, n-k))$.
\end{proof}

{\flushleft\bf Claim 10.} $|N_{A_1}(w_1) \cap N_{A_1}(w_2)|=0$ or 2 for $|V(R)| \geq 2$ where $w_1 \neq w_2 \in V(R)$.
\begin{proof}
First suppose $|N_{A_1}(w_1) \cap N_{A_1}(w_2)| \geq 3$, and take $\{v_1, v_2, v_3\} \subseteq N_{A_1}(w_1) \cap N_{A_1}(w_2)$. Then $G^*$ contains a 6-cycle $C_1:=u^*v_1 w_1 v_2 w_2 v_3u^*$ with chord $u^*v_2$, a contradiction. Next we prove $|N_{A_1}(w_1) \cap N_{A_1}(w_2)| \neq 1$. If equality holds with $N_{A_1}(w_1) \cap N_{A_1}(w_2)=\{v\}$, then $G^*$ contains a 6-cycle $C_2=u^* w_1^{\prime} w_1 v w_2 w_2^{\prime}u^*$ with chord $u^* v$, where $w_1^{\prime} \in N_{A_1}(w_1) \setminus v$ and $w_2^{\prime} \in N_{A_1}(w_2) \setminus v$, a contradiction. We therefore conclude the assertion holds.
\end{proof}

\begin{figure}[http]
\centering
\begin{tikzpicture}[x=1.00mm, y=1.00mm, inner xsep=0pt, inner ysep=0pt, outer xsep=0pt, outer ysep=0pt]
\path[line width=0mm] (108.05,34.94) rectangle +(63.93,45.26);
\definecolor{L}{rgb}{0,0,0}
\definecolor{F}{rgb}{0,0,0}
\path[line width=0.30mm, draw=L, fill=F] (140.02,60.05) circle (0.50mm);
\path[line width=0.30mm, draw=L, fill=F] (156.71,74.08) circle (0.50mm);
\path[line width=0.30mm, draw=L, fill=F] (152.01,74.08) circle (0.50mm);
\path[line width=0.30mm, draw=L, fill=F] (128.50,74.08) circle (0.50mm);
\path[line width=0.30mm, draw=L, fill=F] (133.20,74.08) circle (0.50mm);
\path[line width=0.30mm, draw=L, fill=F] (137.90,74.08) circle (0.50mm);
\path[line width=0.30mm, draw=L, fill=F] (142.60,74.08) circle (0.50mm);
\path[line width=0.30mm, draw=L, fill=F] (117.93,72.39) circle (0.50mm);
\path[line width=0.30mm, draw=L, fill=F] (117.93,68.63) circle (0.50mm);
\path[line width=0.30mm, draw=L, fill=F] (117.93,58.29) circle (0.50mm);
\path[line width=0.30mm, draw=L, fill=F] (117.93,65.93) circle (0.20mm);
\path[line width=0.30mm, draw=L, fill=F] (117.93,63.58) circle (0.20mm);
\path[line width=0.30mm, draw=L, fill=F] (117.93,61.23) circle (0.20mm);
\path[line width=0.15mm, draw=L] (128.48,74.10) -- (140.14,59.93);
\path[line width=0.15mm, draw=L] (133.32,74.10) -- (139.98,60.20);
\path[line width=0.15mm, draw=L] (137.93,74.33) -- (140.25,60.05);
\path[line width=0.15mm, draw=L] (142.54,73.87) -- (140.14,59.82);
\path[line width=0.15mm, draw=L] (151.89,74.33) -- (140.14,59.82);
\path[line width=0.15mm, draw=L] (156.96,74.10) -- (140.49,60.17);
\path[line width=0.15mm, draw=L] (117.93,72.39) -- (139.78,59.93);
\path[line width=0.15mm, draw=L] (117.81,68.39) -- (139.98,59.87);
\path[line width=0.15mm, draw=L] (118.28,58.17) -- (140.25,60.17);
\path[line width=0.15mm, draw=L] (128.62,73.96) -- (133.32,73.96);
\path[line width=0.15mm, draw=L] (138.14,73.96) -- (142.96,73.96);
\path[line width=0.15mm, draw=L] (152.24,73.96) -- (157.06,73.96);
\path[line width=0.15mm, draw=L, fill=F] (115.69,49.12) circle (0.50mm);
\path[line width=0.15mm, draw=L, fill=F] (118.04,49.12) circle (0.50mm);
\path[line width=0.15mm, draw=L, fill=F] (122.75,49.12) circle (0.50mm);
\path[line width=0.15mm, draw=L, fill=F] (127.44,49.12) circle (0.50mm);
\path[line width=0.15mm, draw=L, fill=F] (129.80,49.12) circle (0.50mm);
\path[line width=0.15mm, draw=L, fill=F] (134.50,49.12) circle (0.50mm);
\path[line width=0.15mm, draw=L, fill=F] (139.20,49.12) circle (0.50mm);
\path[line width=0.15mm, draw=L, fill=F] (143.90,49.12) circle (0.50mm);
\path[line width=0.15mm, draw=L, fill=F] (148.60,49.12) circle (0.50mm);
\path[line width=0.15mm, draw=L, fill=F] (153.30,49.12) circle (0.50mm);
\path[line width=0.15mm, draw=L, fill=F] (160.35,49.12) circle (0.50mm);
\path[line width=0.15mm, draw=L, fill=F] (165.05,49.12) circle (0.50mm);
\path[line width=0.15mm, draw=L] (115.93,49.24) -- (139.67,60.05);
\path[line width=0.15mm, draw=L] (118.04,49.00) -- (139.78,60.05);
\path[line width=0.15mm, draw=L] (140.02,59.93) -- (122.75,49.24);
\path[line width=0.15mm, draw=L] (139.90,60.05) -- (127.56,49.24);
\path[line width=0.15mm, draw=L] (139.90,59.93) -- (129.91,49.36);
\path[line width=0.15mm, draw=L] (139.90,60.05) -- (134.50,49.12);
\path[line width=0.15mm, draw=L] (139.90,60.05) -- (139.20,49.36);
\path[line width=0.15mm, draw=L] (140.14,59.93) -- (143.78,49.12);
\path[line width=0.15mm, draw=L] (140.02,60.05) -- (148.48,49.24);
\path[line width=0.15mm, draw=L] (140.02,59.93) -- (153.42,49.36);
\path[line width=0.15mm, draw=L] (139.98,59.87) -- (160.35,49.00);
\path[line width=0.15mm, draw=L] (139.90,60.05) -- (165.17,49.24);
\path[line width=0.15mm, draw=L, fill=F] (118.98,41.01) circle (0.50mm);
\path[line width=0.15mm, draw=L, fill=F] (130.74,41.01) circle (0.50mm);
\path[line width=0.15mm, draw=L] (115.69,48.77) -- (118.98,41.01);
\path[line width=0.15mm, draw=L] (118.04,49.24) -- (118.98,41.01);
\path[line width=0.15mm, draw=L] (122.75,49.24) -- (118.98,41.01);
\path[line width=0.15mm, draw=L] (127.44,49.12) -- (130.85,40.90);
\path[line width=0.15mm, draw=L] (129.80,49.24) -- (130.74,41.01);
\path[line width=0.15mm, draw=L] (134.50,49.12) -- (130.62,40.66);
\path[line width=0.15mm, draw=L, fill=F] (135.44,41.01) circle (0.50mm);
\path[line width=0.15mm, draw=L, fill=F] (137.79,41.01) circle (0.50mm);
\path[line width=0.15mm, draw=L, fill=F] (142.49,41.01) circle (0.50mm);
\path[line width=0.15mm, draw=L, fill=F] (147.19,41.01) circle (0.50mm);
\path[line width=0.15mm, draw=L, fill=F] (149.54,41.01) circle (0.50mm);
\path[line width=0.15mm, draw=L, fill=F] (154.24,41.01) circle (0.50mm);
\path[line width=0.15mm, draw=L, fill=F] (161.29,41.01) circle (0.50mm);
\path[line width=0.15mm, draw=L, fill=F] (163.64,41.01) circle (0.50mm);
\path[line width=0.15mm, draw=L, fill=F] (168.34,41.01) circle (0.50mm);
\path[line width=0.15mm, draw=L] (139.08,49.24) -- (135.55,40.90);
\path[line width=0.15mm, draw=L] (148.48,48.89) -- (147.42,41.48);
\path[line width=0.15mm, draw=L] (139.20,49.12) -- (137.79,41.13);
\path[line width=0.15mm, draw=L] (139.31,49.12) -- (142.60,41.13);
\path[line width=0.15mm, draw=L] (143.78,49.36) -- (135.32,40.66);
\path[line width=0.15mm, draw=L] (143.90,49.24) -- (137.67,40.90);
\path[line width=0.15mm, draw=L] (144.01,49.24) -- (142.49,40.90);
\path[line width=0.15mm, draw=L] (148.60,49.24) -- (149.54,40.90);
\path[line width=0.15mm, draw=L] (148.48,49.12) -- (154.24,41.01);
\path[line width=0.15mm, draw=L] (153.30,49.24) -- (147.30,41.13);
\path[line width=0.15mm, draw=L] (153.42,49.24) -- (149.54,41.25);
\path[line width=0.15mm, draw=L] (153.18,49.12) -- (154.24,41.01);
\path[line width=0.15mm, draw=L] (160.35,49.12) -- (161.29,41.01);
\path[line width=0.15mm, draw=L] (165.05,49.12) -- (161.29,41.01);
\path[line width=0.15mm, draw=L] (160.59,49.00) -- (163.64,41.13);
\path[line width=0.15mm, draw=L] (165.17,49.00) -- (163.76,41.13);
\path[line width=0.15mm, draw=L] (160.47,49.00);
\path[line width=0.15mm, draw=L] (160.59,49.00) -- (168.34,41.13);
\path[line width=0.15mm, draw=L] (165.05,49.00) -- (168.46,41.25);
\path[line width=0.15mm, draw=L] (117.99,64.89) ellipse (3.58mm and 10.71mm);
\path[line width=0.15mm, draw=L] (119.45,49.00) ellipse (5.41mm and 1.41mm);
\path[line width=0.15mm, draw=L] (131.21,49.00) ellipse (5.41mm and 1.41mm);
\path[line width=0.15mm, draw=L] (138.73,41.01) ellipse (5.41mm and 1.41mm);
\path[line width=0.15mm, draw=L] (150.48,41.01) ellipse (5.41mm and 1.41mm);
\path[line width=0.15mm, draw=L] (164.58,41.01) ellipse (5.41mm and 1.41mm);
\path[line width=0.15mm, draw=L, fill=F] (145.31,74.08) circle (0.20mm);
\path[line width=0.15mm, draw=L, fill=F] (147.54,74.08) circle (0.20mm);
\path[line width=0.15mm, draw=L, fill=F] (149.77,74.08) circle (0.20mm);
\path[line width=0.15mm, draw=L, fill=F] (122.98,40.90) circle (0.20mm);
\path[line width=0.15mm, draw=L, fill=F] (125.45,40.90) circle (0.20mm);
\path[line width=0.15mm, draw=L, fill=F] (127.68,40.90) circle (0.20mm);
\path[line width=0.15mm, draw=L, fill=F] (156.12,49.12) circle (0.20mm);
\path[line width=0.15mm, draw=L, fill=F] (157.29,40.90) circle (0.20mm);
\path[line width=0.15mm, draw=L, fill=F] (158.47,40.90) circle (0.20mm);
\path[line width=0.15mm, draw=L, fill=F] (138.96,40.90) circle (0.20mm);
\path[line width=0.15mm, draw=L, fill=F] (139.90,40.90) circle (0.20mm);
\path[line width=0.15mm, draw=L, fill=F] (141.08,40.90) circle (0.20mm);
\path[line width=0.15mm, draw=L, fill=F] (150.71,40.90) circle (0.20mm);
\path[line width=0.15mm, draw=L, fill=F] (151.65,40.90) circle (0.20mm);
\path[line width=0.15mm, draw=L, fill=F] (152.83,40.90) circle (0.20mm);
\path[line width=0.15mm, draw=L, fill=F] (166.93,40.90) circle (0.20mm);
\path[line width=0.15mm, draw=L, fill=F] (164.82,40.90) circle (0.20mm);
\path[line width=0.15mm, draw=L, fill=F] (165.76,40.90) circle (0.20mm);
\path[line width=0.15mm, draw=L, fill=F] (157.29,49.12) circle (0.20mm);
\path[line width=0.15mm, draw=L, fill=F] (156.24,40.90) circle (0.20mm);
\path[line width=0.15mm, draw=L, fill=F] (158.59,49.12) circle (0.20mm);
\draw(142.72,75.84) node[anchor=base west]{\fontsize{8.54}{10.24}\selectfont $t_1$};
\draw(110.35,63.34) node[anchor=base west]{\fontsize{8.54}{10.24}\selectfont $B$};
\draw(142.13,59.46) node[anchor=base west]{\fontsize{8.54}{10.24}\selectfont $u^*$};
\draw(110.05,48.53) node[anchor=base west]{\fontsize{8.54}{10.24}\selectfont $A_1$};
\draw(110.29,40.54) node[anchor=base west]{\fontsize{8.54}{10.24}\selectfont $R$};
\draw(116.75,51.47) node[anchor=base west]{\fontsize{5.69}{6.83}\selectfont $r_1$};
\draw(128.03,51.59) node[anchor=base west]{\fontsize{5.69}{6.83}\selectfont $r_p$};
\draw(136.73,37.37) node[anchor=base west]{\fontsize{5.69}{6.83}\selectfont $s_1$};
\draw(148.48,37.37) node[anchor=base west]{\fontsize{5.69}{6.83}\selectfont $s_2$};
\draw(162.58,37.37) node[anchor=base west]{\fontsize{5.69}{6.83}\selectfont $s_q$};
\draw(117.46,38.43) node[anchor=base west]{\fontsize{5.69}{6.83}\selectfont $w_1$};
\draw(129.21,38.43) node[anchor=base west]{\fontsize{5.69}{6.83}\selectfont $w_p$};
\draw(136.66,50.00) node[anchor=base west]{\fontsize{5.69}{6.83}\selectfont $v_1$};
\draw(140.47,49.88) node[anchor=base west]{\fontsize{5.69}{6.83}\selectfont $v_2$};
\draw(145.89,50.00) node[anchor=base west]{\fontsize{5.69}{6.83}\selectfont $v_3$};
\draw(149.70,49.88) node[anchor=base west]{\fontsize{5.69}{6.83}\selectfont $v_4$};
\draw(133.43,38.50) node[anchor=base west]{\fontsize{5.69}{6.83}\selectfont $w_1^{\prime}$};
\draw(144.54,38.50) node[anchor=base west]{\fontsize{5.69}{6.83}\selectfont $w_2^{\prime}$};
\end{tikzpicture}%
\begin{center}\small{Fig. 3 : The structure of $G^*$, where $t_1\geq 0$,  $p+\sum_{i=1}^{q}=|R|$ and $\sum_{i=1}^{p}r_i+2q=|A_1|$.
}\end{center}
\end{figure}
Note that $d_{A_1}(w_i) \geq 2$ for all $w_i \in R$. From Claims 8-10, we deduce that $G^*$ has the structure shown in Fig. 3. Specifically, for nonnegative integers $p$ and $q$, the induced subgraph $G^*[N_{A_1}(R) \cup R] \cong \bigcup_{i=1}^p K_{1, r_i} \bigcup_{j=1}^q K_{2, s_j}$ and satisfies $k-1+2 t_1+ \sum_{i=1}^p r_i+ \sum_{i=1}^q s_i+p+2 q=n(r_i, s_i \geq 2)$. We now proceed to determine the exact values of $p$ and $q$.

{\flushleft\bf Claim 11.}  $p \leq 1$ and $q \leq 1$.
\begin{proof}
We first prove $p \leq 1$. Suppose, for contradiction, that $p \geq 2$. Then there exist two vertices $w_1, w_2 \in R$ with $G^*[N_{A_1}(w_i) \cup\{w_i\}]\cong K_{1, r_i}$ for $i=1,2$. Without loss of generality, assume that $x_{w_1} \geq x_{w_2}$. Denote by
$$
G^{\prime}=G^*-\sum_{v \in N_{A_1}(w_2)} v w_2+\sum_{v \in N_{A_1}(w_2)} v w_1+u^*w_2+w_1w_2.
$$
Then $\rho(G^{\prime})>\rho(G^*)$ by Lemma \ref{lem::2.4}. However, $G^{\prime}\setminus B$ is also a minimally (2,2)-edge-connected graph since $N_{A_1}(w_1) \cup N_{A_1}(w_2) \cup$ $\{w_1, w_2, u^*\}$ induces a $K_{2, r_1+r_2+1}$ block in $G^{\prime}$, a contradiction.

We now prove $q \leq 1$. Suppose to the contrary that $q \geq 2$. Then $G^*[N_{A_1}(R) \cup R]$ contains induced subgraphs $K_{2, s_1}$ and $K_{2, s_2}$ $(s_1, s_2 \geq 2)$. Let $w_i^{\prime} \in V(K_{2, s_i}) \cap R$ for $i$=1, 2, with $N_{A_1}(w_1^{\prime})=\{v_1, v_2\}$ and $N_{A_1}(w_2^{\prime})=\{v_3, v_4\}$. By symmetry, $x_{v_1}=x_{v_2}$ and $x_{v_3}=x_{v_4}$. Without loss of generality, assume $x_{v_1} \geq x_{v_3}$, implying $x_{v_1}+x_{v_2} \geq x_{v_3}+x_{v_4}$. Let $G^{\prime \prime}=G^*-w_2^{\prime} v_3-w_2^{\prime} v_4+w_2^{\prime} v_1+w_2^{\prime} v_2+v_3v_4$. Clearly, $G^{\prime \prime}\setminus B$ remains minimally $(2, 2)$-edge-connected. Then $\rho(G^{\prime \prime})>\rho(G^*)$ by Lemma \ref{lem::2.4}, which contradicts with the maximality of $\rho(G^*)$.
\end{proof}

{\flushleft\bf Claim 12.} $G^*\cong K(0, n-k)$ for $|V(R)| \geq 2$.
\begin{proof}
Comparing Fig. 1 and Fig. 3 yields $t_2=r_1$ and $t_3=s_1$. Then $G^*\cong K_1(t_1, t_2, t_3)$ by Claim 11, where $t_1, t_3 \geq 0$, and either $t_2=0$ or $t_2 \geq 2$, satisfying
$$n=
\begin{cases}2t_1+t_2+k  & \text { if } t_3=0; \\  
2t_1+t_2+t_3+k+2  & \text { if } t_3\geq 1.
\end{cases}
$$
If $t_3=0$, then $t_2=n\!-\!2t_1\!-\!k$, and thus $G^*\cong K_1(t_1, t_2, 0)=K_1(t_1, n\!-\!k\!-\!2t_1, 0)$. If $t_3\geq 1$, then $t_2+t_3+2=n-2t_1-k$. By Proposition \ref{pro::3.1}, we have 
$$\rho(K_1(t_1, t_2, t_3))<\rho(K_1(t_1, t_2+t_3+1, 0))=\rho(K_1(t_1,n-2t_1-k, 0)).$$
By the maximality of $\rho(G^*)$ again, we get $G^*\cong K_1(t_1, n-k-2t_1, 0)=K(t_1, n-k-2t_1)$. Finally, Proposition \ref{pro::3.2} yields $t_1=0$. Therefore, $G^*\cong K(0, n-k)$.
\end{proof}
From the definitions of $K(0, n-k)$
and $K^{k-2}_{2, n-k}$, we immediately obtain $K(0, n-k)\cong K^{k-2}_{2, n-k}$.
Therefore, combining Claims 9 and 12 completes the proof of Theorem \ref{thm::1}.
\end{proof}

\section{ Proof of Theorem \ref{thm::2}}
Recall that $K^{k-2}_{2, \frac{m-k+2}{2}}$ is a graph obtained from by identifying the $k-2$ degree vertex of $K_{1,k-2}$ and the $\frac{m-k+2}{2}$ degree vertex of $K_{2,\frac{m-k+2}{2}}$. And $K^{k-2}_{2,\frac{m-k-1}{2}}*K_3$ is a graph obtained from by identifying the $k-2$ degree vertex of $K_{1,k-2}$, the vertex of a triangle and the $\frac{m-k-1}{2}$ degree vertex of $K_{2,\frac{m-k-1}{2}}$. 
Let $F^{i}_{\frac{m-i}{3}}$ be obtained from by identifying the $i$ degree vertex of $K_{1,i}$ and the maximal degree vertex of the friend graph with $\frac{m-i}{3}$ triangles. Denote by $v_1$ a maximal degree vertex of the friend graph with $t_1$ triangles, $v_2$ a maximal degree vertex of $K_{2, t_2}$, $v_3$ a vertex of $K_{1,k-2}$ with degree $k\!-\!2$. Let $F(t_1, t_2)$ be the graph obtained from the above three graphs by identifying $v_1, v_2$ and $v_3$.
Prior to proving Theorem \ref{thm::2}, we establish the following lemmas. 

\begin{lem} \label{lem::4.1}
If $m\geq (432k)^2$, then $\rho(K^{k-2}_{2,\frac{m-k+2}{2}})= \frac{\sqrt{2m+2\sqrt{2k^2-2km+m^2-8k+4m+8}}}{2}$ and
$\rho(K^{k-2}_{2,\frac{m-k+2}{2}})> \sqrt{m-k}$.
\end{lem}
\renewcommand\proofname{\bf Proof}
\begin{proof}
The vertex set of $K^{k-2}_{2,\frac{m-k+2}{2}}$ has an equitable partition $\pi$ and its corresponding quotient matrix is
$$B_\pi=\left(
           \begin{array}{cccc}
             0 & 0 & \frac{m-k+2}{2} & k-2 \\
             0 & 0 & \frac{m-k+2}{2} & 0 \\
             1 & 1 & 0 & 0\\
             1 & 0 & 0 & 0\\
           \end{array}
         \right).$$
Then $f(x)=\det(xI_4-B_\pi)=x^4-mx^2+\frac{1}{2}km-\frac{1}{2}k^2-m+2k-2$.  By Lemma \ref{lem::2.1}, $\rho(K^{k-2}_{2,\frac{m-k+2}{2}})= \frac{\sqrt{2m+2\sqrt{2k^2-2km+m^2-8k+4m+8}}}{2}$. Moreover, $\rho(K^{k-2}_{2,\frac{m-k+2}{2}})>\sqrt{m-k}$ since $K^{k-2}_{2,\frac{m-k+2}{2}}$ contains $K_{2,\frac{m-k+2}{2}}$ as a proper spanning subgraph and $\rho(K_{2,\frac{m-k+2}{2}})=\sqrt{m-k+2}$. 
\end{proof}

\begin{lem} \label{lem::4.2}
If $m\geq (432k)^2$, then $\rho(K^{k-2}_{2,\frac{m-k-1}{2}}*K_3)$ is the largest root of $x^5-x^4+(1-m)x^3+(m-3)x^2+(\frac{1}{2}mk-\frac{1}{2}k^2-\frac{1}{2}k)x-1-\frac{k}{2}+m-\frac{mk}{2}+\frac{k^2}{2}=0$ and $\rho(K^{k-2}_{2,\frac{m-k-1}{2}}*K_3)> \sqrt{m-k}$.
\end{lem}
\begin{proof}
The vertex set of $K^{k-2}_{2,\frac{m-k-1}{2}}*K_3$ has an equitable partition $\pi$ and its corresponding quotient matrix is
$$B_\pi=\left(
           \begin{array}{ccccc}
             0 & 0 & \frac{m-k-1}{2} & 2 & k-2 \\
             0 & 0 & \frac{m-k-1}{2} & 0 & 0\\
             1 & 1 & 0 & 0 & 0\\
             1 & 0 & 0 & 1 & 0\\
             1 & 0 & 0 & 0 & 0\\
           \end{array}
         \right).$$
Then $f(x)=\det(xI_5-B_\pi)=x^5-x^4+(1-m)x^3+(m-3)x^2+(\frac{1}{2}mk-\frac{1}{2}k^2-\frac{1}{2}k)x-1-\frac{k}{2}+m-\frac{mk}{2}+\frac{k^2}{2}$, and then $\rho(K^{k-2}_{2,\frac{m-k-1}{2}}*K_3)$ is the largest root of $f(x)=0$ by Lemma \ref{lem::2.1}. Moreover, one can verify that $f'(x)=5x^4-4x^3+3(1-m)x^2+2(m-3)x+\frac{1}{2}mk-\frac{1}{2}k^2-\frac{1}{2}k> 0$ and  $f(\sqrt{m-k})<0$.
Thus, $\rho(K^{k-2}_{2,\frac{m-k-1}{2}}*K_3)>\sqrt{m-k}$. 
\end{proof}

\begin{lem} \label{lem::4.3}
If $m\geq (432k)^2$, then $\rho(F^{k-3}_{\frac{m-k+3}{3}})$ is the largest root of $x^3-x^2+(1-\frac{2}{3}m-\frac{1}{3}k)x+k-3=0$ and $\rho(F^{k-3}_{\frac{m-k+3}{3}})<\sqrt{m-k}$. 
\end{lem}
\begin{proof}
The vertex set of $F^{k-3}_{\frac{m-k+3}{3}}$ has an equitable partition $\pi$ and its corresponding quotient matrix is
$$B_\pi=\left(
           \begin{array}{ccc}
             0 & k-3 & \frac{2(m-k+3)}{3} \\
             1 & 0  & 0\\
             1 & 0 & 1 \\
           \end{array}
         \right).$$
Then $f(x)=\det(xI_3-B_\pi)=x^3-x^2+(1-\frac{2}{3}m-\frac{1}{3}k)x+k-3.$ By Lemma \ref{lem::2.1}, we can get $\rho(F^{k-3}_{\frac{m-k+3}{3}})$ is the largest root of $f(x)=0$. Moreover, one can verify that $f'(x)=3x^2-2x+1-\frac{2}{3}m-\frac{1}{3}k> 0$
and $f(\sqrt{m-k})>0$ for $m\geq (432k)^2$. Thus, $\rho(F^{k-3}_{\frac{m-k+3}{3}})<\sqrt{m-k}$. 
\end{proof}

\renewcommand\proofname{\bf Proof of Theorem \ref{thm::2}}
\begin{proof}
Suppose that $G^*$ is a graph that attains the maximum spectral radius among all minimally $(k,k)$-edge-connected graphs with size $m$. Denoted by $B=\{v \mid d(v)=1\}$ and $|B|=b$. Let $x$ be the Perron vector of $A(G^*)$ with coordinate $x_{u^*}=\max\{x_i \mid i\in V(G^*)\}$. We analyze the structure of $G^*$ through two cases:

{\flushleft\bf{Case 1.}} $b\leq k-3$.
 
Let $A=V(G^*)\setminus (B\cup{u^*})$. Suppose $A$ contains $l$ blocks, denoted $A_1$, $A_2, \ldots$, $A_l$. By Lemma \ref{lem::2.7}, each induced subgraph $G^*[A_i]$ is minimally $(k-b,k-b)$-edge-connected for $1\leq i\leq l$.
 
{\flushleft\bf Claim 1.} $A_i$ is isomorphic to a cycle.
\renewcommand\proofname{\bf Proof}
\begin{proof}
We equivalently show that $d_{A_i}(u)=2$ for all $u\in V(A_i)$. Suppose otherwise that some block $A_g$ is not a cycle. Then there exists $u\in V(A_g)$ with $d_{A_g}(u)\geq 3$, and  by Corollary \ref{cor::1}, $A_g$ contains at least three internal paths of length at least $k-b$. Let $P:=u_1v_1v_2\cdots v_lu_2$ be an internal path in $A_g$ such that there is still a cycle containing $u^*$ in $A_g\setminus P$, implying $u_1\neq u^*$ or $u_2\neq u^*$. We define $G_1=G^*-u_1v_1-u_2v_l+u^*v_1+u^*v_l$. Notice that $d_{G_1}(u_1)\geq 2$ and $d_{G_1}(u_2)\geq 2$. Then $G_1[A_g]$ is a minimally $(k-b, k-b)$-edge-connected graph clearly. In addition, the block $C:=u^*v_1v_2\cdots v_lu^*$ of $G_1$ is also minimally $(k-b, k-b)$-edge-connected. Thus, $G_1$ is minimally $(k, k)$-edge-connected. Since $x_{u^*}=\max\{x_i \mid i\in V(G^*)\}$, Lemma \ref{lem::2.4} implies $\rho(G_1)> \rho(G^*)$, a contradiction.
\end{proof}

{\flushleft\bf Claim 2.} $G^*$ has exactly one cut vertex $u^*$.
\begin{proof}
If not, $G^*$ either contains no cut vertices, or contains at least two cut vertices. If $G^*$ has no cut vertices, then $G^*\cong C_m$. Let $C_m:=u_1u_2\cdots u_mu_1$ and define $G'=G^*-\{u_{m-1}u_m\}+\{u_{m-1}u_1\}$. By symmetry, $x_{u_i}=u_1$ for $2\leq i\leq m$. Furthermore, $G'$ remains minimally $(k,k)$-edge-connected for $m\geq (432k)^2$. However, $\rho(G')>\rho(G^*)$ by Lemma \ref{lem::2.4}, a contradiction. Thus, $G^*$ must contain cut vertices.
Suppose there exists another cut vertex $u$ of $G^*$ with $u^*\in C$, where $C$ is a component of $G^*-u$. For any $v\in N(u)\setminus N(u^*)$ with $v\notin V(C)$, we have $G_2=G^*-\{uv\}+\{u^*v\}$ is minimally $(k,k)$-edge-connected. Consequently,  $\rho(G_2)>\rho(G^*)$ by $x_{u^*}=\max\{x_i \mid i\in V(G')\}$ and Lemma \ref{lem::2.4}, a contradiction. Thus, $G^*$ has exactly one cut vertex $u^*$.
\end{proof}
Therefore, $G^*\!-\!u^*$ must be the union of some isolated vertices and some paths. 

{\flushleft\bf Claim 3.} $b= k-3$.
\begin{proof}
Suppose to the contrary, $b\leq k-4$. According to Corollary \ref{cor::1}, each internal path in $A_i$ has length at least $k-b$ where $1\leq i\leq l$, implying $|V(A_i)|\geq 4$. Let $A_i=u^*u_1\cdots u_tu^*$, where $t\geq 3$, and let $G_3=G^*-u_{t-2}u_{t-1}+u_{t-2}u_{t}-u_{t-1}u_{t}+u_{t-1}u^*$. Then $G_3$ is also minimally $(k,k)$-edge-connected because $|B|=b+1\leq k-3$ and $G_3[A_i]$ is a minimally $(k-b-1,k-b-1)$-edge-connected graph. Note that $x_{u_t}> x_{u_{t-1}}$ by Lemma \ref{lem::2.8}. Thus, $\rho(G_3)> \rho(G^*)$ by Lemma \ref{lem::2.4}, a contradiction. 
\end{proof}

{\flushleft\bf Claim 4.} $|A_i|= 3$ for any $1\leq i\leq l$.
\begin{proof}
If not, there exists block $A_r$ $(1\leq r\leq l)$ with $|A_r|=t\geq 4$. Let $A_r=u^*u_1\ldots u_{t-1}u^*$ and define $G_4=G^*-\{u_1u_2\}+\{u^*u_2\}$. Then $G_4$ is a minimally $(k,k)$-edge-connected graph because $|B|=k-2$ and $G_4[A_r]$ is minimally $(2,2)$-edge-connected. 
So, $\rho(G_4)>\rho(G^*)$ by Lemma \ref{lem::2.4}, a contradiction. 
\end{proof}

Combining Claims $1$-$4$, we can obtain that $l=\frac{m-k+3}{3}$, which yields that $G^*\cong F^{k-3}_{\frac{m-k+3}{3}}$. By Lemma \ref{lem::4.3}, we have $\rho(F^{k-3}_{\frac{m-k+3}{3}})<\sqrt{m-k}$. If $m-k\equiv 0\pmod 2$, then $\rho(K^{k-2}_{2,\frac{m-k+2}{2}})> \sqrt{m-k}$ by Lemma \ref{lem::4.1} and $K^{k-2}_{2,\frac{m-k+2}{2}}$ is a minimally $(k,k)$-edge-connected graph, a contradiction. If $m-k\equiv 1\pmod 2$, then $\rho(K^{k-2}_{2,\frac{m-k-1}{2}}*K_3)> \sqrt{m-k}$ by Lemma \ref{lem::4.2} and $K^{k-2}_{2,\frac{m-k-1}{2}}*K_3$ is a minimally $(k,k)$-edge-connected graph, a contradiction.

{\flushleft\bf{Case 2.}} $b= k-2$.

We first assert that $B\subset N(u^*)$. If not, there exists $vw\in E(G^*)$ such that $d_{G^*}(v)=1$ and $w\neq u^*$. Let $G'=G^*+\{vu^*\}-\{vw\}$. Then $\rho(G')>\rho(G^*)$ by $x_{u^*}=\max\{x_i \mid i\in V(G^*)\}$ and Lemma \ref{lem::2.4}, a contradiction. Thus, the vertex $u^*$ is a cut vertex of $G^*$ and $G^*[u^*\cup B]\cong K_{1,k-2}$. Denoted by $K$ the induce subgraph of $u^*\cup B$ and $H=G^*- B$. 

{\flushleft\bf Claim 5.} $G^*$ has exactly one cut vertex $u^*$.
\begin{proof}
If not, there exists an additional cut vertex $u$ of $G^*$ with $u^*\in C$, where $C$ is a component of $G^*-u$. For any vertex $v\in N(u)\setminus N(u^*)$ and $v\notin V(C)$. Let $G_1=G^*-\{uv\}+\{u^*v\}$. It is evident that $G_1$ is also a minimally $(k,k)$-edge-connected graph. However, $\rho(G_1)>\rho(G^*)$ by $x_{u^*}=\max\{x_i \mid i\in V(G^*)\}$ and Lemma \ref{lem::2.4}. This contradiction establishes the assertion.
\end{proof}

{\flushleft\bf Claim 6.} $\frac{m+k}{2}+1 \leq|V(G^*)| \leq \frac{m+k}{2}+\sqrt{m}-1$.
\begin{proof}
For $m-k\equiv 0\pmod 2$, the graph $K^{k-2}_{2,\frac{m-k+2}{2}}$ is minimally $(k,k)$-edge-connected and $\rho(G^*) \geq \rho(K^{k-2}_{2,\frac{m-k+2}{2}})$. For $m-k\equiv 1\pmod 2$, the graph $K^{k-2}_{2,\frac{m-k-1}{2}}*K_3$ is minimally $(k,k)$-edge-connected and $\rho(G^*) \geq \rho(K^{k-2}_{2,\frac{m-k-1}{2}}*K_3)$. Thus, $\rho^2(G^*) \geq m-k$ by Lemmas \ref{lem::4.1} and \ref{lem::4.2}. We have $m-k\leq\rho^2(H)+\rho^2(K)$ by Lemma \ref{lem::2.3}. Additionally, $\rho^2(K) =k-2$ since $K\cong K_{1,k-2}$. Then $\rho^2(H)\geq m-2k+2$. Note that $e(H)=m-k+2$ and $\delta(H)\geq2$. We can deduce that $\rho(H)\leq\frac{1}{2}+\sqrt{2(m-k+2)-2|V(H)|+\frac{9}{4}}$ by Lemma \ref{lem::2.5}, which gives us $|V(H)|\leq \frac{m-k+2}{2}+\frac{1}{2}\sqrt{m-2k+2}+1$ and 
\begin{align*}
|V(G^*)|&=|V(H)|+|V(K)|-1 \\
&\leq (\frac{m-k+2}{2}+\frac{1}{2}\sqrt{m-2k+2}+1)+(k-1)-1 \\
&\leq \frac{m+k}{2}+\sqrt{m}-1.
\end{align*}
On the other hand, $H $ is minimally $(2, 2)$-edge-connected by Lemma \ref{lem::2.7}. Combining this with Lemma \ref{lem::2.10}, we find that $m-k+2 \leq 2(|V(H)|-2)$, that is, $|V(H)|\geq\frac{m-k+6}{2}$. As a result, we have $|V(G^*)|=|V(H)|+|V(K)|-1\geq \frac{m+k}{2}+1$.
\end{proof}
Let $\alpha=\frac{1}{36 k}$ and $U_0=\{v \in V(G^*): x_v>\alpha x_{u^*}\}$. For convenience, we denote $\rho^*=\rho(G^*)$. Since $\rho^*x_v=x_{u^*}$ holds for any vertex $v\in B$, this implies that $B\cap U_0=\emptyset$. Therefore, $U_0\subseteq V(H)$.

{\flushleft\bf Claim 7.} $|U_0|<2 \sqrt{m}$.
\begin{proof}
Recall that $\rho^* \geq \sqrt{m- k}$. Let $v$ be an arbitrary vertex in $U_0$.  According to the definition of $U_0$, it can be known that $\rho^* x_v>\sqrt{m- k} \cdot \alpha x_{u^*}$. Furthermore, 
$$
\rho^* x_v=\sum_{v \in N(v) \cap U_0} x_v+\sum_{v \in N(v) \setminus U_0} x_v \leq(d_{U_0}(v)+\alpha \cdot d_{V(G^*) \setminus U_0}(v)) x_{u^*} .
$$
Thus, $\alpha \sqrt{m- k}<d_{U_0}(v)+\alpha \cdot d_{V(G^*) \setminus U_0}(v)$. By summing this inequality over all vertices $v \in U_0$, then
$$
\alpha \sqrt{m- k} \cdot|U_0|<2 e(U_0)+\alpha \cdot e(U_0, V(G^*) \setminus U_0) \leq 4|U_0|+\alpha m,
$$
where the last inequality follows from $e(U_0) \leq 2|U_0|$ by Lemma \ref{lem::2.10}. Moreover, note that $m\geq (432k)^2$. Hence, $4|U_0| \leq \frac{\alpha}{3} \sqrt{m-k} \cdot|U_0|$. This leads us to conclude that
$$
\alpha \sqrt{m- k} \cdot|U_0|<\frac{\alpha}{3} \sqrt{m-k} \cdot|U_0|+\alpha m.
$$
Thus, $|U_0| \leq \frac{3 m}{2 \sqrt{m-k}}<2 \sqrt{m}$.
\end{proof}
Let $U=\{v \in V(G^*): x_v \geq \frac{1}{4} x_{u^*}\}$. Then $u^* \in U$ and $U \subseteq U_0$. Claim 7 yields additional properties of $U$.

{\flushleft\bf Claim 8.} $d(v)>(\frac{x_v}{x_{u^*}}-\frac{1}{6 k}) \frac{m}{2}$ for any $v \in U$.
\begin{proof}
Let $v_0 \in U$ be an arbitrary vertex and let us consider the case where $x_{v_0}=\beta x_{u^*}$ with $\frac{1}{4} \leq \beta \leq 1$. In the following, we show that $d(v_0)>(\beta-\frac{1}{6 k}) \frac{m}{2}$.
Define $M_0=V(G^*) \setminus\{v_0\}$. It is evident that $x_v \leq \alpha x_{u^*}$ for all $v \in M_0 \setminus U_0$. Thus, 
\begin{equation}\label{equ::2}
\begin{aligned}
\rho^{*2} x_{v_0} & =d(v_0) x_{v_0}+\sum_{v \in M_0 \backslash U_0} d_{N(v_0)}(v) x_v+\sum_{v \in M_0 \cap U_0} d_{N(v_0)}(v) x_v \\
& \leq(\beta \cdot d(v_0)+\sum_{v \in M_0 \backslash U_0} \alpha \cdot d_{N(v_0)}(v)+\sum_{v \in M_0 \cap U_0} d_{N(v_0)}(v)) x_{u^*}.
\end{aligned}
\end{equation}
Since $N(v_0) \subseteq M_0$, we can conclude that
\begin{equation}\label{equ::3}
\begin{aligned}
\sum_{v \in M_0 \setminus U_0} d_{N(v_0)}(v) \leq \sum_{v \in M_0} d_{M_0}(v)=2 e(M_0) \leq 2 m.
\end{aligned}
\end{equation}
Note that $v_0 \in U \subseteq U_0$ and $M_0 \cap U_0=U_0 \setminus\{v_0\}$. We have $\sum_{v \in M_0 \cap U_0} d_{N(v_0)}(v)=$ $\sum_{v \in U_0} d_{N(v_0)}(v)-d(v_0)$. On the other hand, $\sum_{v \in U_0} d_{N(v_0)}(v) \leq e(U_0)+e(U_0 \cup N(v_0))$. There $e(U_0) \leq 2|U_0|$ and $e(U_0 \cup N(v_0)) \leq 2|U_0|+d(v_0)$ by Lemma \ref{lem::2.10}. Thus,
\begin{equation}\label{equ::4}
\begin{aligned}
\sum_{v \in M_0 \cap U_0} d_{N(v_0)}(v) \leq 4|U_0|+ d(v_0).
\end{aligned}
\end{equation}
By substituting equations (\ref{equ::3}) and (\ref{equ::4}) into (\ref{equ::2}) and dividing both sides by $x_{u^*}$, we conclude that
$$
\rho^{*2} \beta \leq(\beta+1) d(v_0)+2 \alpha m+4|U_0|.
$$
Note that $|U_0|<2 \sqrt{m}$ by Claim 7. Furthermore, $\rho^{*2} \geq m- k$ according to Lemmas \ref{lem::4.1} and \ref{lem::4.2}, where $m\geq (432k)^2$. It follows that  $k \beta+4|U_0| \leq  \alpha m$. Consequently,
$$
(m-k)\beta \leq(\beta+1) d(v_0)+2\alpha m+4|U_0|,
$$
which implies that $d(v_0) \geq \frac{\beta-3 \alpha}{\beta+1} m$. It suffices to show $\frac{2(\beta-3 \alpha)}{\beta+1}>(\beta-\frac{1}{6 k}),$
or equivalently, $f_1(\beta) \triangleq \beta^2-\frac{\beta}{6 k}-\beta<0$ as $\alpha=\frac{1}{36 k}$. Recall that $\frac{1}{4} \leq \beta \leq 1$. It is easy to verify that $f_1(\beta)|_{\max }=f_1(1)=-\frac{1}{6 k}<0$. Hence, $d(v)>(\frac{x_v}{x_{u^*}}-\frac{1}{6 k}) \frac{m}{2}$ for every vertex $v \in U$.
\end{proof}

{\flushleft\bf Claim 9.} $x_v>(1-\frac{1}{6 k}) x_{u^*}$ for any $v \in U$.
\begin{proof}
By the definition of $U$, we observe that $x_v \geq \frac{1}{4} x_{u^*}$ for each $v \in U$. Now, suppose to the contrary that there exists a vertex $v_0 \in U$ such that $x_{v_0}=\beta x_{u^*}$, where $\frac{1}{4} \leq \beta \leq 1-\frac{1}{6 k}$. Denote $M=V(G^*)$, then
\begin{equation}\label{equ::5}
\begin{aligned}
\rho^{*2} x_{u^*}=\sum_{v \in M} d_{N(u^*)}(v) x_v=\sum_{v \in M \setminus U_0} d_{N(u^*)}(v) x_v+\sum_{v \in U_0} d_{N(u^*)}(v) x_v.
\end{aligned}
\end{equation}
Recall that $x_v \leq \alpha x_{u^*}$ for each $v \in M \setminus U_0$. Thus, 
\begin{equation}\label{equ::6}
\begin{aligned}
\sum_{v \in M \backslash U_0} d_{N(u^*)}(v) x_v \leq \sum_{v \in M} d_M(v) \cdot \alpha x_{u^*}=2 \alpha m x_{u^*}.
\end{aligned}
\end{equation}
By applying Lemma \ref{lem::2.10} and Claim 7, we find that $e(U_0) \leq 2|U_0|<4 \sqrt{m}$. This leads to the following result:
\begin{equation}\label{equ::7}
\begin{aligned}
\sum_{v \in U_0} d_{N(u^*)}(v)=e(U_0, N(u^*) \setminus U_0)+2 e(U_0) \leq m+4 \sqrt{m}.
\end{aligned}
\end{equation}
Next, denote by $\overline{N}(u^*)$ the set of vertices not adjacent to $u^*$, that is, $\overline{N}(u^*)=M \setminus N[u^*]$. Based on Claim 8, we know that $|N(u^*)| \geq(1-\frac{1}{6 k}) \frac{m}{2}$. Additionally, $|M| \leq \frac{m+k-2}{2}+\sqrt{m}$ by Claim 6. Given that $m\geq (432k)^2$, we determine that
$$
|\overline{N}(u^*)|=|M|-|N(u^*)| \leq \frac{m+k-2}{2}+\sqrt{m}-(1-\frac{1}{6 k}) \frac{m}{2}\leq\sqrt{m}+\frac{m}{11k}<\frac{m}{10k}.
$$
According to Claim 8, we also have $d(v_0)>(\beta-\frac{1}{6 k}) \frac{m}{2}$. Hence,  $d_{N(u^*)}(v_0) \geq d(v_0)-|\overline{N}(u^*)|>(\beta-\frac{11}{30 k}) \frac{m}{2}$. Notice that $v_0 \in U \subseteq U_0$ and $x_{v_0}-x_{u^*}=(\beta-1) x_{u^*}$. In view of (7), then
\begin{equation}\label{equ::8}
\begin{aligned}
\sum_{v \in U_0} d_{N(u^*)}(v) x_v & \leq \sum_{v \in U_0} d_{N(u^*)}(v) x_{u^*}+d_{N(u^*)}(v_0)(x_{v_0}-x_{u^*}) \\
& <(m+4 \sqrt{m}+(\beta-1)(\beta-\frac{11}{30  k}) \frac{m}{2}) x_{u^*}.
\end{aligned}
\end{equation}
Note that $\rho^*>\sqrt{m- k}$ by Lemmas \ref{lem::4.1} and \ref{lem::4.2}. Combining with (\ref{equ::5}), (\ref{equ::6}) and (\ref{equ::8}), we can obtain
\begin{equation}\label{equ::9}
\begin{aligned}
(m- k) x_{u^*}<\rho^{*2} x_{u^*}<(2 \alpha m+m+4 \sqrt{m}+(\beta-1)(\beta-\frac{11}{30 k}) \frac{m}{2}) x_{u^*}.
\end{aligned}
\end{equation}
Recall that $m\geq (432k)^2$ and $\alpha=\frac{1}{36 k}$. Then $ k+4 \sqrt{m} \leq \alpha m$. In view of (\ref{equ::9}), then
\begin{equation}\label{equ::10}
\begin{aligned}
(\beta-1)(\beta-\frac{11}{30 k})+6\alpha >0,
\end{aligned}
\end{equation}
where $6 \alpha=\frac{1}{6 k}$. Let $f_2(\beta)=(\beta-1)(\beta-\frac{11}{30 k})+\frac{1}{6 k}$, where $\frac{1}{4} \leq \beta \leq 1-\frac{1}{6 k}$. It is easy to check that $f_2(\beta)|_{\max }=f_2(\frac{1}{4})$. Consequently, we always have $f_2(\beta)<0$, which contradicts (\ref{equ::10}). Therefore, the proof is complete.
\end{proof}

{\flushleft\bf Claim 10.} $|U|=2$.
\begin{proof}
By Claims 8 and 9, we can conclude that  $d(u)>(1-\frac{1}{3 k}) \frac{m}{2}$ for every vertex $u \in U$. In particular, this means that $d(u^*)>(1-\frac{1}{6 k}) \frac{m}{2}$.
We first consider the case where $|U| \geq 3$. Let us select three vertices $u_1, u_2,  u_{3} \in U$. Then there holds
\begin{equation}\label{equ::11}
\begin{aligned}
|\bigcap_{i=1}^{3} N(u_i)| \geq \sum_{i=1}^{3}|N(u_i)|-2|\bigcup_{i=1}^{3} N(u_i)| \geq3(1-\frac{1}{6 k}) \frac{m}{2}-2|V(G^*)|.
\end{aligned}
\end{equation}
Where $2|V(G^*)| \leq m+k-2+ 2\sqrt{m}$ since $|V(G^*)| \leq \frac{m+k-2}{2}+\sqrt{m}$ by Claim 6. In view of (11), we can see that $|\cap_{i=1}^{3} N(u_i)| >\frac{m}{2}-\frac{m}{3k}>3$ for $m\geq (432k)^2$. Thus, $G^*$ contains a copy of $K_{3, 3}$. However,  $\kappa'_2(G^*\setminus B-e)\geq 2$ for any edge $e\in K_{3, 3}$, which contradicts Lemma \ref{lem::2.6}. Therefore, we deduce that every minimally $(2,2)$-edge-connected graph cannot contain $K_{3, 3}$ as a subgraph, a contradiction. Hence, $|U| \leq 2$.

Next, we consider the case where $|U| = 1$, that is, $U=\{u^*\}$. Denote $M=V(G^*)$. By the definition of $U$, we know that $x_u<\frac{1}{4} x_{u^*}$ for every $u \in M \setminus U$. Noting that
$$
\rho^{*2} x_{u^*}=\sum_{u \in M} d_{N(u^*)}(u) x_u \leq d(u^*)x_{u^*}+\frac{1}{4}(2(m-d(u^*)))x_{u^*}=\frac{1}{2}(m+d(u^*))x_{u^*}.
$$
On the other hand, $\rho^{*2}> m-k$. Then $m-k\leq \frac{1}{2}(m+d(u^*))$. Furthermore, $d(u^*)< |V(G^*)|<\frac{m+k-2}{2}+\sqrt{m}$, then $\frac{1}{2}(m+\frac{m+k-2}{2}+\sqrt{m})<m-k$ for $m\geq (432k)^2$, a contradiction. Thus, $|U|=2$.
\end{proof}
According to Claim 10, we may assume that $U=\{u^*, u_1\}$. Now, denote by $V$ the common neighbourhood of vertices in $U$, that is, $V=\{v: v\in N(u^*) \cap N(u_1)\}$. Moreover, let $W=V(G^*) \setminus(U \cup V \cup B)$ in the following.

{\flushleft\bf Claim 11.} $G^*[U \cup V] \cong K_{2,  |V|}$ and $e(V, W)=0$.
\begin{proof}
Recall that $|G^*|<\frac{m+k-2}{2}+\sqrt{m}$, $d(u^*) \geq(1-\frac{1}{6 k}) \frac{m}{2}$ and $d(u_1) \geq(1-\frac{1}{3 k}) \frac{m}{2}$. Then $|V|\geq d(u^*)+d(u_1)-|G^*|\geq (1-\frac{1}{6 k}) \frac{m}{2}+(1-\frac{1}{3 k}) \frac{m}{2}-(\frac{m+k-2}{2}+\sqrt{m})=\frac{m}{2}-\frac{m}{4k}-\sqrt{m}-\frac{k}{2}+1>\frac{m}{2}-\frac{m}{3k}$. Therefore, $G^*[U \cup V]$ contains a spanning subgraph isomorphic to $K_{2, |V|}$, where $|V| \geq 3$. We can conclude that $G^*[U \cup V] \cong K_{2, |V|}$. Now choose an arbitrary vertex $v \in V$. Clearly, $v$ is a vertex of degree $2$ in $G^*[U \cup V]$.

Recall that $G^*-B$ is minimally $(2,2)$-edge-connected. We assert that $v$ is either a cut vertex or a vertex of degree $2$ in $G^*$. Furthermore, since $v$ is not a cut vertex of $G^*$ by Claim 5, we only need to prove that $d_{G^*}(v)= 2$. If not, $d_{G^*}(v)\geq 3$. Then $e(v,W)> 0$. Let $G=G^*-u^*v$. Then $\kappa'_2(G)=2$ since $d_{G}(v')\geq 3$ for any $v'\in G-B$ and $v$ is not a cut vertex of $G$. Hence, $d_{G^*}(v)= 2$, which also implies $e(V, W)=0$.
\end{proof}

{\flushleft\bf Claim 12.} $e(U, W)=e(u^*, W)$.
\begin{proof}
If not, $e(U, W)>e(u^*, W)$, which implies $e(u_1, W)\geq 1$. Suppose that $u_1w_1\in E(G^*)$. We have $u^*w_1\notin E(G^*)$ because $w_1\notin V$. Assuming that $u_1w_1\in B_1$, where $B_1$ is a block of $G^*$. Now let us consider the longest path $P:=w_1w_2\cdots w_p$ in $W$ that contains $w_1$. 
If $w_p=w_1 \in W$, then $w_1u_1$ is a cut edge, which contradicts the fact that $B_1$ is a block. If $w_p \in N_W(u_1)$, this implies that $u_1$ is a cut vertex, which contradicts the fact that  $B_1$ is a block. If $w_p \in N_W(u^*)$, then let $G'=G^*-w_1u_1+w_1u^*$. We have $G'$ is still minimally $(k,k)$-edge-connected and $\rho(G')>\rho(G^*)$ by Lemma \ref{lem::2.4}, a contradiction. 
\end{proof}

{\flushleft\bf Claim 13.} $\{u^*\} \cup W$ induces some cycles with a common vertex $u^*$.
\begin{proof}
Suppose to the contrary that some leaf block $A$ in $\{u^*\} \cup W$ is not a cycle. Then there exists some vertex $w_1\in A$ such that $d(w_1)\geq3$. Note that $A$ does not contain a chord by Lemma \ref{lem::2.9}. Thus, $A$ must contain an internal path $P:=w_1w_2\cdots w_l$ of length at least $2$. And hence, $w_1\neq u^*$ or $w_l\neq u^*$. Without loss of generality, suppose $w_l\neq u^*$. Let $G'=G^*-w_1w_2-w_{l-1}w_l+u^*w_2+u_1w_{l-1}$. Then $G'$ is still a minimally $(k,k)$-edge-connected graph and $x_{u_1}>x_{w_l}$ since $u_1\in U$. However, $\rho(G')>\rho(G^*)$ by Lemma \ref{lem::2.4}. This contradiction establishes $\{u^*\} \cup W$ induces some cycles with a common vertex $u^*$. 
\end{proof}

{\flushleft\bf Claim 14.} $G^*[\{u^*\} \cup W]\cong F_{\frac{|W|}{2}}$. 
\begin{proof}
If not, there exists some cycle $C_l:=u^*w_1w_2\cdots w_{l-1}u^*$ of length at least $4$ in $G^*[\{u^*\} \cup W]$. If $l=4$, then $u^*w_1w_2w_3$ forms a quadrilateral. Let $G_1=G^*-w_1w_2-w_2w_3+w_1u_1+w_3u_1$. Then $G_1$ is still a minimally $(k,k)$-edge-connected graph. We have $\rho(G_1)>\rho(G^*)$ by Lemma \ref{lem::2.4} again, a contradiction.
If $l\geq 5$, then let $G_2=G^*-\{w_1w_2\}-\{w_2w_3\}+\{w_3u^*\}$. Similar to the above contradiction. Therefore, $G^*[\{u^*\} \cup W]\cong F_{\frac{|W|}{2}}$. 
\end{proof}
By the maximality of $\rho(G^*)$ again, we get $G^*\cong F(t_1,\frac{m-3t_1-k+2}{2})$. At last, we will show $t_1=0$ for $m-k\equiv0\pmod 2$ and $t_1=1$ for $m-k\equiv1\pmod 2$.

{\flushleft\bf Claim 15.} $t_1\leq 1$.
\begin{proof}
If not, $t_1\geq 2$. Suppose $w_1w_2\in E(W)$ and $w_3w_4\in E(W)$.
Then let $G=G^*-\{w_4u^*\}-\{w_1w_2\}-\{w_3w_4\}+\{w_iu\mid 1\leq i\leq3\}$. Clearly, $G$ is still a minimally $(k,k)$-edge-connected graph. Recall that 
$x$ is the Perron vector of $A(G^*)$, and $\rho^*=\rho(G^*)$. By symmetry, we have $x_{w_i}=x_{w_1}$ for $2\leq i\leq 4$. Thus from $A(G^*)x=\rho^*x$, we obtain $x_{w_1}=\frac{x_{u^*}}{\rho^*-1}$. Note that $x_{u}>(1-\frac{1}{6 k}) x_{u^*}$ by Claim 9.
Combining with $\rho(G^*)\geq \sqrt{m-k}$, then
\begin{align*}
\rho(G)-\rho(G^*)&\geq x^T\big(A(G)-A(G^*)\big)x\\
&\geq2(3x_{w_1}x_u-x_{w_1}x_{u^*}-2x_{w_1}^2)\\
&\geq2x_{w_1}(3x_u-x_{u^*}-\frac{2x_{u^*}}{\rho^*-1})~~(\text{since}~ x_{w_1}=\frac{x_{u^*}}{\rho^*-1})\\
&\geq2x_{w_1}(3(1-\frac{1}{6 k}) x_{u^*}-x_{u^*}-\frac{2x_{u^*}}{\rho^*-1})~~ (\text{since}~ x_{u}>(1-\frac{1}{6 k}) x_{u^*})\\
&>0~~ (\text{since}~ \rho*\geq \sqrt{m-k}).
\end{align*}
It follows that $\rho(G)>\rho(G^*)$, a contradiction.
\end{proof}
If $m-k\equiv0\pmod 2$, the above claim means that $t_1=0$, then $G^*\cong F(0,\frac{m-k+2}{2})$. If $m-k\equiv1\pmod 2$, the above claim means that $t_1=1$, then $G^*\cong F(1,\frac{m-k-1}{2})$.
Notice that $F(0,\frac{m-k+2}{2})\cong K^{k-2}_{2,\frac{m-k+2}{2}}$ for $m-k\equiv0\pmod 2$, and $K^{k-2}_{2,\frac{m-k-1}{2}}*K_3\cong F(1,\frac{m-k-1}{2})$ for $m-k\equiv1\pmod 2$. It completes the proof of Theorem \ref{thm::2}.
\end{proof}

\section{ Acknowledgement}
This work is supported by Natural Science Foundation of Xinjiang Uygur Autonomous Region (No. 2024D01C41), Tianshan Talent Training Program (No. 2024TSYCCX0013), the Basic scientific research in universities of Xinjiang Uygur Autonomous Region (No. XJEDU2025P001), National Natural Science Foundation of China (No.\ 12361071) and  Excellent Doctor Innovation program of Xinjiang University (No. XJU2024 \ BS043).

\end{document}